\documentclass[12pt, reqno]{amsart}%

\usepackage{amssymb}
\usepackage{amsmath,esint}
\usepackage[shortlabels]{enumitem}
\usepackage{amsthm}
\usepackage{amscd,mdwlist}
\usepackage[margin=1in]{geometry}
\usepackage{color}
\usepackage{cite}
\usepackage{bbm}
 \usepackage{tcolorbox}
 \usepackage{thmtools}
 \usepackage[backgroundcolor=white, bordercolor=blue,
 linecolor=blue]{todonotes}
 
\usepackage[pagebackref, pdfpagelabels, plainpages=false]{hyperref}

\hypersetup{
   colorlinks=true,
   linkcolor=red,
   filecolor=magenta,
   urlcolor=cyan,
   citecolor=blue,
}

\theoremstyle{plain}
\declaretheorem[title=Theorem, parent=section]{theorem}
\declaretheorem[title=Lemma,sibling=theorem]{lemma}
\declaretheorem[title=Proposition,sibling=theorem]{proposition}
\declaretheorem[title=Corollary,sibling=theorem]{corollary}

\theoremstyle{definition}
\declaretheorem[title=Definition,sibling=theorem]{definition}
\declaretheorem[title=Remark,sibling=theorem]{remark}
\declaretheorem[title=Remark, numbered=no]{remark*}
\declaretheorem[title=Example, sibling=theorem]{example}

\numberwithin{equation}{section}

\DeclareMathOperator{\dist}{dist}
\DeclareMathOperator{\supp}{supp}

\newcommand{\il}{\int\limits}
\newcommand{\iil}{\iint\limits}
\DeclareMathOperator{\R}{\mathbb{R}}
\renewcommand{\d}{\mathrm{d}}

\newcommand{\VnuOm}{V_{\nu}(\Omega|\R^d)}

\newcommand{\eps}{\varepsilon}
\newcommand{\bet}[1]{\left| #1 \right|}

\newcommand{\vertiii}[1]{{\left\vert\kern-0.25ex\left\vert\kern-0.25ex\left\vert #1 \right\vert\kern-0.25ex\right\vert\kern-0.25ex\right\vert}}

\usepackage[english]{babel}

\addto\extrasenglish{%

}

\begin{document}

\title{Mosco convergence of nonlocal to local quadratic forms}

\author{Guy Fabrice Foghem Gounoue}
\address{Universit\"at Bielefeld, Fakult\"at f\"ur Mathematik, Postfach 100131, 33501 Bielefeld, Germany}
\email{guy.gounoue@math.uni-bielefeld.de}

\author{Moritz Kassmann} 
\address{Universit\"at Bielefeld, Fakult\"at f\"ur Mathematik, Postfach 100131, 33501 Bielefeld, Germany}
\email{moritz.kassmann@uni-bielefeld.de}

\author{Paul Voigt}
\address{Universit\"at Bielefeld, Fakult\"at f\"ur Mathematik, Postfach 100131, 33501 Bielefeld, Germany}
\email{pvoigt@math.uni-bielefeld.de}

\thanks{Financial support by the DFG via IRTG 2235: ``Searching for the regular in the irregular: Analysis of singular and random systems'' is gratefully acknowledged.}

\begin{abstract}
We study sequences of nonlocal quadratic forms and function spaces that are related to Markov jump processes in bounded domains with a Lipschitz boundary. Our aim is to show the convergence of these forms to local quadratic forms of gradient type. Under suitable conditions we establish the convergence in the sense of Mosco. Our framework allows bounded and unbounded nonlocal operators to be studied at the same time. Moreover, we prove that smooth functions with compact support are dense in the nonlocal function spaces under consideration.
\end{abstract}



\maketitle
\date{December 29, 2018}

\section{Introduction} 

In the last two decades the study of nonlocal operators and integro-differential operators has attracted much attention. Here, we have in mind linear or nonlinear operators satisfying a maximum principle as the fractional Laplace operator does. In this work we study the convergence of sequences of such nonlocal operators to local differential operators. Let $(\alpha _n)$ be a sequence of numbers $\alpha_n \in (0,2)$ with $\lim \alpha_n = 2$. Given a function $u \in C^\infty_c(\R^d)$, the convergence 
\begin{align}\label{eq:frac-laplace_to_laplace}
(-\Delta)^{\alpha_n/2} u \longrightarrow  -\Delta u 
\end{align}
clearly holds true. There are many possible ways resp. topologies in which the operators $(-\Delta)^{\alpha_n/2}$ converge to the classical Laplace operator. In this work we do not study the operators directly. We focus on corresponding quadratic forms because they appear naturally when formulating boundary or complement value problems. Note that for functions $u,v \in C^\infty_c(\R^d)$ the equality
\begin{align}\label{eq:part-int_frac-laplace}
\il_{\R^d}  (-\Delta)^{\alpha/2} u (x) v (x) \d x  = \frac{C_{d, \alpha}}{2} \iil_{\R^d \R^d} \frac{(u(x)-u(y))^2}{|x-y|^{d+\alpha}} \d x \, \d y
\end{align}
holds true. Here,  $ C_{d, \alpha}$ is a constant depending on the dimension $d$ and the value $\alpha \in (0,2)$, for which the relation $\widehat{(-\Delta)^{\alpha/2} u(\xi)} = |\xi|^\alpha \widehat{u}(\xi)$ holds true in  $C^\infty_c(\R^d)$. Let us mention that asymptotically $C_{d, \alpha}\asymp \alpha (2-\alpha)$, which is important for our analysis. Interested readers  may consult \cite{Hitchhiker} for  more details about the fractional Laplacian $(-\Delta)^{\alpha/2}$ and the constant $C_{d,\alpha}$. If $\Omega \subset \R^d$ is open and $u \in C^\infty_c(\Omega)$, then one can easily show 

\begin{align}\label{eq:frac-forms_to_gradient-form}
\frac{C_{d, \alpha}}{2} \iil_{\Omega \Omega} \frac{(u(x)-u(y))^2}{|x-y|^{d+\alpha}} \d x \, \d y \longrightarrow \il_{\Omega} |\nabla u|^2 \quad \text{ as } \alpha \to 2- \,.
\end{align}
In light of equalities \eqref{eq:frac-laplace_to_laplace} and \eqref{eq:part-int_frac-laplace} this is a natural result. A more interesting version of this result is proved in \cite{BBM01}. Therein, it is shown that \eqref{eq:frac-forms_to_gradient-form} holds true if $\Omega \subset \R^d$ is a bounded open set with a Lipschitz boundary and $u \in H^1(\Omega)$. The regularity assumption on $\Omega$ and $u$ ensures that suitable extensions of $u$ to the whole space exist. Analogously to the above, one can easily prove for $\Omega \subset \R^d$ open and $u \in C^\infty_c(\R^d)$ the following result:
\begin{align}
\frac{C_{d, \alpha}}{2} \iil_{\Omega \R^d} \frac{(u(x)-u(y))^2}{|x-y|^{d+\alpha}} \d x \, \d y &\longrightarrow \il_{\Omega} |\nabla u|^2 \quad \text{ as } \alpha \to 2- \,, \label{eq:new-frac-forms_to_gradient-form} \\
\frac{(2+d)}{\eps^{d+2}\omega_{d-1}} \iil_{\Omega \R^d} \big(u(x)-u(y) \big)^2 \mathbbm{1}_{B_\eps}(x-y) \d x \, \d y &\longrightarrow \il_{\Omega} |\nabla u|^2 \quad \text{ as } \eps \to 0+ \,. \label{eq:bounded-frac-forms_to_gradient-form}
\end{align}
The expression on the left-hand side of \eqref{eq:new-frac-forms_to_gradient-form} naturally appears when studying nonlocal Dirichlet or Neumann problems with prescribed data on the complement of $\Omega$, see \cite{FKV15 ,DRV17}.  It is important for the study of Dirichlet-to-Neumann maps of certain nonlocal problems involving the the fractional Laplacian, see \cite{Calderon-Tuhin}. The expression also appears when studying extension theorems for nonlocal operators, see \cite{DyKa18, BGPR17}. 

\medskip

Assertions \eqref{eq:frac-forms_to_gradient-form}, \eqref{eq:new-frac-forms_to_gradient-form} and \eqref{eq:bounded-frac-forms_to_gradient-form} describe the convergence of a sequence of numbers since the function $u$ is fixed. The main aim of this paper is to prove a result in the spirit of \eqref{eq:new-frac-forms_to_gradient-form} but not for a given function. We study the convergence of forms in the sense of Mosco, see \autoref{def:mosco}, which is a well-known generalization of the famous concept of $\Gamma$-convergence. The result  then applies to variational solutions to boundary data or complement data problems. Note that our main result \autoref{thm:Mosco-convergence} covers sequences of forms with bounded and unbounded kernels at the same time. 

\medskip

An important role in our study is played by function spaces. We assume that $\Omega$ is a bounded open subset of $\R^d$. For several results we assume that $\Omega$ has a Lipschitz boundary. Let us introduce generalized Sobolev-Slobodeckij-like spaces with respect to an unimodal L\'{e}vy measure $\nu(h) \d h$. Assume $\nu:\R^d\setminus \{0\} \to [0, \infty)$ is a radial function, which (a) satisfies $\nu \in L^1(\R^d, (1 \wedge |h|^2) \d h)$ and (b) is almost decreasing, i.e., there is $c \geq 0$ such that $|y| \geq |x|$ implies $\nu(y) \leq c \nu(x)$. The function $\nu$ then is the density of an unimodal L\'evy measure. Possible examples are given by $\nu(h) = \mathbbm{1}_{B_1}(h)$ and for $\alpha \in (0,2)$ by $\nu(h) = C_{d,\alpha} |h|^{-d-\alpha}$ for $h \in \R^d, h \ne 0$. With the help of $\nu$ we can now define several function spaces. Set 
\begin{align*}
H_{\nu} (\Omega)= \Big\{u \in L^2(\Omega)| \iil_{\Omega\Omega} \big(u(x)-u(y) \big)^2 \, \nu (x-y)\mathrm{d}x\,\mathrm{d}y<\infty \Big \}\,.             
\end{align*}
We endow this space with the norm
\begin{align*}
\|u\|^2_{H_{\nu} (\Omega)}= \|u\|^2_{L^{2} (\Omega)}+ \iil_{\Omega\Omega} \big(u(x)-u(y) \big)^2 \, \nu (x-y) \mathrm{d}x\,\mathrm{d}y.               
\end{align*}
Note that for bounded functions $\nu$, e.g., in the case $\nu(h) = \mathbbm{1}_{B_1}(h)$, the space $H_{\nu} (\Omega)$ equals $L^2(\Omega)$.  Following \cite{FKV15}  we define $ V_{\nu} (\Omega|\mathbb{R}^d)$ as follows:
\begin{align*}
V_{\nu} (\Omega|\mathbb{R}^d) = \Big\lbrace  u: \R^d \to \R \text{ meas. } |  \, [u]^2_{V_{\nu} (\Omega|\mathbb{R}^d)} = \!\!\iil_{(\Omega^c\times \Omega^c)^c}  \!\!\big(u(x)-u(y) \big)^2 \, \nu (x-y) \mathrm{d}x \, \mathrm{d}y <\infty  \Big\rbrace \,.              
\end{align*}
We endow this space with two norms as follows:
\begin{align*}
\|u\|^2_{V_{\nu} (\Omega|\mathbb{R}^d)} &:= \|u\|^2_{L^{2} (\R^d)}+  \iil_{(\Omega^c\times \Omega^c)^c} \big(u(x)-u(y) \big)^2 \, \nu (x-y) \mathrm{d}x \, \mathrm{d}y \,, \\               
\vertiii{u}^2_{V_{\nu} (\Omega|\mathbb{R}^d)} &:= \|u\|^2_{L^{2} (\Omega)}+  \iil_{(\Omega^c\times \Omega^c)^c} \big(u(x)-u(y) \big)^2 \, \nu (x-y) \mathrm{d}x \, \mathrm{d}y \,.
\end{align*}

Note that for $\alpha \in (0,2)$, $\nu(h) = C_{d,\alpha} |h|^{-d-\alpha}$ for $h \in \R^d, h \ne 0$, the space $H_{\nu} (\Omega)$ equals the classical Sobolev-Slobodeckij space $H^{\alpha/2}(\Omega)$. For the same choice of $\nu$ we define $V^{\alpha/2} (\Omega|\mathbb{R}^d)$ as the space $V_{\nu} (\Omega|\mathbb{R}^d)$. Our first main result concerns the density of smooth functions in $\VnuOm$. Its rather technical proof is provided in \autoref{sec:spaces}.

\begin{theorem}\label{thm:density}
	Assume $\Omega\subset \mathbb{R}^d$ is open, bounded and $\partial \Omega$  is Lipschitz continuous. Let $\nu$ be as above. Then $C_c^\infty(\mathbb{R}^d)$ is dense in $V_{\nu}(\Omega|\mathbb{R}^d)$ with respect to the two norms mentioned above, i.e. for $u \in V_{\nu}(\Omega|\mathbb{R}^d)$ there exists a sequence $(u_n) \subset C_c^\infty(\mathbb{R}^d)$ with 
	\begin{align*}
	\|u_n-u\|_{V_{\nu}(\Omega|\mathbb{R}^d)}  \longrightarrow 0 \text{ as } n\to\infty \,.
	\end{align*} 
	Obviously, the convergence $\vertiii{u_n-u}_{V_{\nu}(\Omega|\mathbb{R}^d)} \to 0$ follows.
\end{theorem}

\noindent Next, let us explain for which sequences of nonlocal quadratic forms we can prove convergence to a classical local gradient form.  

\begin{definition}\label{def:nu-alpha} Let $(\rho_\varepsilon)_{0<\varepsilon<2}$ be a family of radial functions approximating the Dirac measure at the origin. We assume that every $\eps,\delta > 0$
	\begin{align*}
	\begin{split}
	\rho_\varepsilon\geq 0, \quad \int_{\mathbb{R}^d}\rho_\varepsilon (x)\d x=1, \quad \lim_{\varepsilon\to 0^+}\int_{|x|>\delta}\rho_\varepsilon (x)\d x=0\,.
	\end{split}
	\end{align*}
Moreover, we assume that $h \mapsto |h|^{-2}\rho_\eps(h)$ is almost decreasing, i.e., for some $c \geq 1$ and all $x,y$ with $|x| \leq |y|$ we have $|y|^{-2} \rho_\eps(y) \leq c \, |x|^{-2} \rho_\eps(x)$. Given a sequence $(\rho_\eps)_{0<\varepsilon<2}$ with the aforementioned properties, we define a sequence $(\nu^\alpha)_{0<\alpha<2}$ of functions $\nu^\alpha:\R^d \setminus\{0\} \to \R$ by  $\nu^\alpha (h) = |h|^{-2} \rho_{2-\alpha}(h)$. This sequence is used to set up function spaces below.  
\end{definition}

\begin{example}\label{ex:most-important}
For $\rho_\eps (h) = \frac{\eps}{\omega_{d-1}} |h|^{-d+\eps} \mathbbm{1}_{B_1}(h)$ we obtain $\nu^\alpha (h) = \frac{2-\alpha}{\omega_{d-1}} |h|^{-d-\alpha} \mathbbm{1}_{B_1}(h)$ and $H_{\nu^\alpha}(\Omega) = H^{\alpha/2}(\Omega)$. Note that there is no sequence $(\rho_\eps)$ satisfying the conditions above, for which $\nu^\alpha (h) = C_{d,\alpha} |h|^{-d-\alpha}$ for all $h$. One would need to relax the integrability condition on $\rho_\eps$. Consequently, the vector spaces $V_{\nu^\alpha}(\Omega|\R^d)$ and $V^{\alpha/2} (\Omega|\mathbb{R}^d)$ do not coincide. However, the normed space $(V_{\nu^\alpha}(\Omega|\R^d),\|\cdot\|_{V_{\nu^\alpha} (\Omega|\mathbb{R}^d)})$ is equivalent to the normed space $(V^{\alpha/2} (\Omega|\mathbb{R}^d),\|\cdot\|_{V^{\alpha/2} (\Omega|\mathbb{R}^d)})$, where
\[\|u\|^2_{V^{\alpha/2} (\Omega|\mathbb{R}^d)} := \|u\|^2_{L^{2} (\R^d)}+  \iil_{(\Omega^c\times \Omega^c)^c} \frac{\big(u(x)-u(y) \big)^2}{\bet{x-y}^{d+\alpha}}  \mathrm{d}x \, \mathrm{d}y \,.\]
\end{example}	

\begin{example}\label{ex:bounded-nu}
As the following example shows, $(\nu^\alpha)$ can be a sequence of bounded functions. For $\eps \in (0,2)$ define $\rho_\eps$ by 
\begin{align}\label{eq:def-rho-bd-nu}
\rho_\varepsilon(h) = \frac{d+2}{ \omega_{d-1}\varepsilon^{d+2}} |h|^{2}\mathbbm{1}_{B_{\varepsilon}}(h) \qquad (h \in \R^d)\,.
\end{align} 
Define $\nu^{\alpha}$ for $\alpha \in (0,2)$ as in \autoref{def:nu-alpha}. Then for every $\alpha \in (0,2)$ $H_{\nu^\alpha}(\Omega)$ is equivalent to $L^{2}(\Omega)$ and $(V_{\nu^\alpha}(\Omega|\R^d),\|\cdot\|_{V_{\nu^\alpha} (\Omega|\mathbb{R}^d)})$ is equivalent to $L^{2}(\R^d)$. Note that these equivalences are not uniform in $\alpha$. 
\end{example}

\noindent Note that each function $\nu^\alpha$ determines a symmetric unimodal  L\'{e}vy measure, i.e., it is a radially almost decreasing function and $\min(1, |h|^2) \in L^1(\R^d, \nu^\alpha(h) \d h)$. Next, let us introduce the nonlocal bilinear forms under consideration. We recall that $\Omega \subset \mathbb{R}^d $ is an open bounded set. Given $\alpha \in (0,2)$, $J^\alpha: \mathbb{R}^d\times \mathbb{R}^d \setminus \operatorname{diag} \to [0, \infty]$ and sufficiently smooth functions $u,v: \mathbb{R}^d\to \mathbb{R}$, we define 
\begin{align}
\mathcal{E}^{\alpha}_{\Omega}(u,v) &=  \iil_{\Omega \Omega} \big(u(y)-u(x)\big) \big(v(y)-v(x)\big) J^\alpha(x,y)\d x \, \d y\,, \label{eq-inner-form} \\
\mathcal{E}^{\alpha}(u,v) &= \iil_{(\Omega^c\times \Omega^c)^c} \big(u(y)-u(x)\big) \big(v(y)-v(x) J^\alpha(x,y) \d x \, \d y, \label{eq-ext-form}
\end{align}

\noindent In the sequel we will not introduce a separate notation for the quadratic forms $u \mapsto \mathcal{E}^{\alpha}_{\Omega}(u,u)$ and $u \mapsto \mathcal{E}^{\alpha}(u,u)$. Note that $(\Omega^c\times \Omega^c)^c$ equals $(\Omega\times \Omega)  \cup (\Omega\times \Omega^c) \cup  (\Omega^c\times \Omega)$. We assume that $(J^\alpha)_{0<\alpha<2}$ is a sequence of  positive symmetric kernels $J^\alpha: \mathbb{R}^d\times \mathbb{R}^d \setminus \operatorname{diag} \to [0, \infty]$ satisfying the following conditions:

\medskip

\begin{itemize}
	\item[(E)] There exists a constant $\Lambda\geq 1$ such that for every $\alpha\in  (0,2)$ and all $x,y \in \mathbb{R}^d$, $x \ne y$, with $|x-y|\leq 1$ 
\begin{align}\label{eq:elliptic-condition}\tag{$E$}
\Lambda^{-1} \nu^\alpha (x-y) \leq J^\alpha(x,y) &\leq  \Lambda \nu^\alpha (x-y) 
\end{align}
	\item[(L)] For every $\delta >0$
 \begin{align} \label{eq:integrability-condition}\tag{$L$}
\lim_{\alpha \to 2^-}\sup_{x\in \mathbb{R}^d} \int_{|h| > \delta} J^\alpha(x,x+h)dh=0.
\end{align}
\end{itemize}%

\noindent Finally, let us define the limit object, which is a local quadratic form of gradient type. Given $x \in \R^d$ and $\delta > 0$, we define the symmetric matrix $A(x) = (a_{ij}(x))_{1\leq i,j\leq d}$ by
\begin{align}\label{eq:coef-matrix}
a_{ij}(x) = \lim_{\alpha\to 2^{-}} \int_{B_\delta}  h_ih_j J^\alpha(x,x+h)dh
\end{align}
and for $u,v \in H^{1}(\Omega)$ the corresponding bilinear form by 
\begin{align*}
\mathcal{E}^A(u,v):=  \il_{\Omega} \big( A(x)\nabla u(x), \nabla v(x) \big)  \d x \,. 
\end{align*}

\noindent Conditions \eqref{eq:elliptic-condition} and \eqref{eq:integrability-condition} are sufficient in order to show convergence results similar to \eqref{eq:new-frac-forms_to_gradient-form} and \eqref{eq:bounded-frac-forms_to_gradient-form}, see \autoref{thm:quadratic-convergence-BBM}. As we will see in \autoref{prop:elliptic-matrix}, conditions \eqref{eq:elliptic-condition} and \eqref{eq:integrability-condition} ensure that the symmetric matrices $A(\cdot)$ defined in  \eqref{eq:coef-matrix} are  uniformly positive definite and bounded. For our main result, \autoref{thm:Mosco-convergence}, we impose translation invariance of the kernels:
\begin{itemize}
\item[(I)] For each $\alpha \in (0,2)$ the kernel $J^\alpha$ is translation invariant, i.e., for every $h \in \mathbb{R}^d$
\begin{align}\label{eq:translation-invariance}\tag{I}
J^\alpha(x+h, y+h) = J^\alpha(x, y) 
\end{align}
\end{itemize}

\medskip

\begin{remark}
	(i) Under conditions \eqref{eq:elliptic-condition} and \eqref{eq:integrability-condition} the expression $\int_{B_\delta}  h_ih_j J^{\alpha_n}(x,x+h)dx$ converges for a suitable subsequence of $(\alpha_n)$. The existence of the limit in \eqref{eq:coef-matrix} poses an implicit condition on the family $(J^\alpha)$. (ii) \eqref{eq:elliptic-condition} and \eqref{eq:integrability-condition} ensure that the quantity $a_{ij}(x)$  does not  depend on the choice of $\delta$ and is bounded as a function in $x$. (iii) Under condition \eqref{eq:translation-invariance} the functions $a_{ij}(x)$ are constant in $x$. 
\end{remark}

\noindent Let us formulate our second main result.

\begin{theorem}\label{thm:Mosco-convergence}
	Let $\Omega\subset \mathbb{R}^d$ be an open bounded set with a Lipschitz continuous boundary. Assume \eqref{eq:elliptic-condition}, \eqref{eq:integrability-condition} and \eqref{eq:translation-invariance}.
	Then the two families of  quadratic forms $(\mathcal{E}^\alpha_{\Omega}(\cdot, \cdot), H_{\nu^\alpha}( \Omega) )_{\alpha}$  and  
	$( \mathcal{E}^\alpha(\cdot, \cdot),V_{\nu^\alpha}( \Omega|\mathbb{R}^d) )_{\alpha}$ both converge to  $( \mathcal{E}^A(\cdot,\cdot), H^{1}( \Omega) )$ in the Mosco sense in $L^2(\Omega)$ as $\alpha\to 2^-$.
\end{theorem}

\medskip

A stronger version of \autoref{thm:Mosco-convergence} not assuming condition \eqref{eq:translation-invariance} will be proved elsewhere, see also \cite{Voi17}. We refer the reader to \autoref{def:mosco} for details about the Mosco convergence of bilinear forms. Note that \autoref{thm:quadratic-convergence-BBM}, which is part of the proof of \autoref{thm:Mosco-convergence}, implies the convergence results \eqref{eq:frac-forms_to_gradient-form}, \eqref{eq:new-frac-forms_to_gradient-form} and \eqref{eq:bounded-frac-forms_to_gradient-form} for fixed functions $u$.

\medskip

Let us discuss the assumption on the family $(J^\alpha)_\alpha$ and provide some examples. Condition \eqref{eq:elliptic-condition} is a sufficient condition for what can be seen as nonlocal version of the classical ellipticity condition for second order operators in divergence form. Condition \eqref{eq:integrability-condition} ensures that long-range interactions encoded by $J^\alpha(x,y)$ vanish as $\alpha\to 2^-$. As a result, for some $\alpha_0\in (0,2)$, the quantity  
\begin{align} \label{eq:consequence-integrability}
\kappa_0 = \sup_{\alpha \in(\alpha_0,2)}\sup_{x\in \mathbb{R}^d} \int_{|h|>1} J^\alpha(x,x+h)dh
\end{align}
is finite. One can easily check that conditions \eqref{eq:elliptic-condition} and \eqref{eq:integrability-condition} imply the following uniform L\'evy integrability type property:
\begin{align*}
\sup_{\alpha \in(\alpha_0,2)}\sup_{x\in \mathbb{R}^d} \int_{\mathbb{R}^d} (1\land |h|^2)J^\alpha(x,x+h) \d h<\infty \,.
\end{align*}

\begin{example}\label{ex:J-guys-singular} For $\eps > 0$ set $\rho_{\eps}(h) = \frac{\eps}{\omega_{d-1}} |h|^{-d+\eps}\mathbbm{1}_{B_1}(h)$. Define $\nu^{\alpha}$ for $\alpha \in (0,2)$ as in \autoref{def:nu-alpha}. Then conditions \eqref{eq:elliptic-condition}, \eqref{eq:integrability-condition} and \eqref{eq:translation-invariance} are fulfilled for each of the following cases and $\beta>0$:
	\begin{align*}
	J_1^{\alpha}(x,y) &= C_{d, \alpha} |x-y|^{-d-\alpha} \,, \\
	J_2^{\alpha}(x,y) &=  C_{d, \alpha} |x-y|^{-d-\alpha}\mathbbm{1}_{B_1}(x-y)+ (2-\alpha)|x-y|^{-d-\beta} \mathbbm{1}_{\mathbb{R}^d\setminus B_1}(x-y)\,, \\
	J_3^{\alpha}(x,y) &=  C_{d, \alpha} |x-y|^{-d-\alpha}\mathbbm{1}_{B_1}(x-y)+ (2-\alpha)J(x,y) \mathbbm{1}_{\mathbb{R}^d\setminus B_1}(x-y) \,,
	\end{align*}
	where $J$ is a symmetric function satisfying $\sup_{x\in \mathbb{R}^d} \int_{\mathbb{R}^d\setminus B_\delta} J(x,x+h)dh<\infty$  for every $\delta >0$. Regarding \autoref{thm:Mosco-convergence}, in the cases   $J_1^\alpha$, $J_2^\alpha$ and $J_3^\alpha$, one obtains $A(x) = (\delta_{ij})_{1\leq i,j\leq d}$, i.e., the matrix $A$ equals the identity matrix.   
	
\end{example}

\begin{example}
In \autoref{ex:J-guys-singular} we provide examples of singular kernels $J$. As we explain above, \autoref{thm:Mosco-convergence} applies to bounded kernels, too. Here is one example. For $\eps \in (0,2)$ define $\rho_\eps$ as in \eqref{eq:def-rho-bd-nu}. Define $\nu^{\alpha}$ for $\alpha \in (0,2)$ as in \autoref{def:nu-alpha}. Then conditions \eqref{eq:elliptic-condition}, \eqref{eq:integrability-condition} and \eqref{eq:translation-invariance} are fulfilled for $J_4^{\alpha}(x,y) = \frac{1}{(2-\alpha)^{d+2}}\mathbbm{1}_{B_{2-\alpha}}(x-y)$. As in the cases above, in the case $J_4$ one obtains $A(x) = (\delta_{ij})_{1\leq i,j\leq d}$. We refer the reader to \autoref{sec:examples} for more examples. 
\end{example}

\medskip

Let us relate our result to other works. We study \autoref{thm:density} as a tool needed for the proof of \autoref{thm:Mosco-convergence}. However, the density result itself is of importance for the study of nonlocal problems in bounded domains. We refer the reader to \cite{DyKa18, BGPR17, KaWa18} for recent results involving function spaces of the type of $\VnuOm$.

\medskip

\autoref{thm:Mosco-convergence} is closely related to the weak  convergence of the finite-dimensional distributions  of stochastic processes. Since both quadratic forms, $(\mathcal{E}^\alpha_{\Omega}(\cdot, \cdot), H_{\nu^\alpha}( \Omega) )_{\alpha}$  and  \linebreak
$( \mathcal{E}^\alpha(\cdot, \cdot),V_{\nu^\alpha}( \Omega|\mathbb{R}^d) )_{\alpha}$ turn out to be regular Dirichlet forms, cf. \autoref{cor:regular-DF}, they correspond to L\'evy processes. In dependence of the choice of $\nu^\alpha$, the L\'evy measure has finite mass or not. \autoref{thm:Mosco-convergence} implies that the distributions of these processes converge weakly to the distribution of a diffusion process defined by the Dirichlet form $(\mathcal{E}^A(\cdot,\cdot), H^{1}( \Omega) )$. In \cite{Mos94} (see also \cite{KS03}) it is shown that Mosco convergence of a sequence of symmetric closed forms is equivalent to the 
convergence of the sequence of associated semigroups (or of associated resolvents) and implies the weak convergence the finite-dimensional distributions of the corresponding processes if any.  Note that several authors have studied the weak convergence of Markov processes with the help of Dirichlet forms, e.g., in \cite{LyZh96, KuUe97, MRZ98, Sun98, Kol05, Kol06, BBCK09, CKK13}. Most of related results are concerned with situations where the type of the process does not change, i.e., diffusions converge to a diffusion or jump processes converge to a jump process. In the present work, we consider examples where a sequence of jump processes in bounded domains converges to a diffusion. This will appear implicitly as  consequence of the  Mosco convergence in \autoref{thm:Mosco-convergence}. 

\medskip

The Dirichlet form $(\mathcal{E}^\alpha_{\Omega}(\cdot, \cdot), H_{\nu^\alpha}( \Omega) )_{\alpha}$ has appeared in the analysis literature for decades. When $\nu^\alpha$ is singular, then it arises naturally through the norms of Sobolev-Slobodeckij spaces introduced by Aronszajn, Gagliardo and Slobodeckij. The regular Dirichlet form generates a censored jump process, which is introduced and thoroughly studied in \cite{BBC03}. Jumps from $\Omega$ into $\R^d \setminus \Omega$ are erased from the underlying free jump process. The stochastic process is restarted each time such a jump occurs.  The situation is very different for the Dirichlet form  $( \mathcal{E}^\alpha(\cdot, \cdot),V_{\nu^\alpha}( \Omega|\mathbb{R}^d) )_{\alpha}$. It appears in \cite{DRV17} in connection with the study of nonlocal problems with Neumann-type conditions, see also \cite{LMPPS18}. The function space $\VnuOm$ is central for the Hilbert space approach to complement value problems with Dirichlet data in \cite{FKV15}. The article \cite{DRV17} offers some probabilistic interpretation but a mathematical study of the corresponding stochastic process seems not to be available yet. The authors have been informed that, in an ongoing project Z. Vondracek addresses the probabilistic interpretation of  quadratic forms including examples like $( \mathcal{E}^\alpha(\cdot, \cdot),V_{\nu^\alpha}( \Omega|\mathbb{R}^d) )_{\alpha}$. Of course, reflections of jump processes have been studied for a long time, e.g. in \cite{MeRo85}. 

\medskip

In the case of bounded jump measures $\nu^\alpha$ the works on so-called nonlocal diffusion equations study similar problems, cf. \cite{CERW07, AMRT10}. Bounded kernels also appear in the study of peridyamics. Neumann boundary conditions have recently been studied in this context in \cite{AC17, TTD17}. Last, let us mention that integro-differential operators have been considered by several authors with nonlocal Neumann conditions in the framework of strong solutions or viscosity solutions, cf. \cite{GaMe02, BCGJ14}. 

\medskip

The paper is organized as follows. In \autoref{sec:spaces} we study the function spaces $\VnuOm$ in detail. In particular, we prove that the subspace $C^\infty_c(\R^d)$ is dense in $V_{\nu}(\Omega|\mathbb{R}^d)$. \autoref{sec:main-proof} is devoted to the proof of \autoref{thm:Mosco-convergence}. 

\bigskip

\emph{Acknowledgement:} The authors thank Vanja Wagner (Zagreb) for helpful discussions on the proof of \autoref{thm:density}.

\section{Density of smooth functions}\label{sec:spaces}

The aim of this section is to prove \autoref{thm:density}. Let us recall the corresponding setup. $\Omega$ is a bounded open subset of $\R^d$ with a Lipschitz boundary. The function $\nu:\R^d\setminus \{0\} \to [0, \infty)$ is radial and satisfies $\nu \in L^1(\R^d, (1 \wedge |h|^2) \d h)$. Moreover, it is almost decreasing, i.e., there is $c \geq 0$ such that $|y| \geq |x|$ implies $\nu(y) \leq c \nu(x)$. The space $\VnuOm$ is defined as above.    

\medskip

First, let us explain why, for certain choices of $\nu$, it is natural to consider the norm $\vertiii{\cdot}_{\VnuOm}$ on the space $\VnuOm$.

\begin{proposition}\label{prop:natural-norm-on-V}
Assume $\nu $ is given as above. 

\smallskip

\noindent (a) If  $\Omega\subset B_{|\xi|/2}(0) $   for some $ \xi \in \mathbb{R}^d$ with $\nu(\xi)\neq 0$.  Then $\VnuOm \subset L^2(\Omega)$.\\
(b) Assume $\nu$ is positive on  sets of positive measure, i.e. $\nu$ has full support. Then there exists another  almost decreasing radial measure $\widetilde{\nu}: \mathbb{R}^d \to [0,\infty) $ and  a constant $C>0$  both depending only on $\nu, d$ and $\Omega$ such that

\begin{enumerate}[(i)]
	\item $\widetilde{\nu}(\mathbb{R}^d)<\infty$\,,
	\item $0\leq \widetilde{\nu} \leq C(1\land \nu)$\,,
	\item $\VnuOm \subset L^2(\R^d,\widetilde{\nu}(h)\d h) \subset L^1(\R^d,\widetilde{\nu}(h)\d h)$\,,
	\item on $\VnuOm$, the  norms $\vertiii{\cdot}_{\VnuOm}$ and  $\vertiii{\cdot}^*_{\VnuOm}$ with 
	\begin{align*}
	\vertiii{u}^{*2}_{\VnuOm}=\int_{\mathbb{R}^d} u^2(x) \widetilde{\nu}(x)\d x+ \iil\limits_{(\Omega^c\times \Omega^c)^c} (u(x)-u(y))^2\nu(x-y)\d x\d y
	\end{align*}
	are equivalent. 
\end{enumerate}
\end{proposition}

\begin{remark} Regarding property (ii) let us mention that in some cases like $\nu(h) = |h|^{-d-\alpha}$ it is possible to obtain $\widetilde{\nu} \asymp 1\land \nu$. In the aforementioned case one could define $\widetilde{\nu}(h) = (1+|h|)^{-d-\alpha}$ for $h \in \R^d$.
\end{remark}

\begin{proof}
	First, if  $\Omega\subset B_{|\xi|/2}(0)$, then  for all $x,y\in \Omega$ we have $\nu(x-y)\geq c'$ with $c'= c\nu(\xi)>0 $. By Jensen's inequality, we have 
	
	\begin{align*}
	\iil_{(\Omega^c\times \Omega^c)^c} \big(u(x)-u(y)\big)^2 \nu(x-y) \d x \, \d y&\geq c' \iil_{\Omega\Omega} (|u(x)|-|u(y)|)^2\d x \, \d y\\
	&\geq c'|\Omega| \il_{\Omega} \Big(|u(x)|-\hbox{$\fint_{\Omega}|u|$}\Big)^2\d x.
	\end{align*}
	This shows that the mean value  $\fint_{\Omega}|u|$ is finite. We conclude $u \in L^2(\Omega)$ because of 
	\[\int_{\Omega}u^2(x) \leq 2\int_{\Omega} \Big(|u(x)|-\hbox{$\fint_{\Omega}$}|u|\Big)^2+ 2|\Omega| \Big(\hbox{$\fint_{\Omega}|u|$}\Big)^2\,. \]
	
The proof of part (b) is similar to the proof of \cite[Proposition 13]{DyKa18}. Assume $\nu $ has full support. Since $\Omega$ is bounded, there is $R\geq 1$ large enough such that $\Omega\subset B_R(0)$. Clearly, we have $|x-y|\leq R(1+|y|)$ for all $x\in \Omega$ and all $y\in \R^d$. The monotonicity condition on $\nu$ implies $\nu(R(1+|y|))\leq c\nu(x-y) $. Set $\widetilde{\nu}(h) = \nu(R(1+|h|))$ for $h \in \R^d$, where we abuse the notation and write $\nu(|y|)$ instead of $\nu(y)$  for $y\in \R^d$. Let us show that $\widetilde{\nu}$ satisfies the desired conditions. Note that $(ii)$ is a direct consequence of the fact that $|h|\leq R(1+|h|)$  and $R\leq R(1+|h|)$ for all $h \in \R^d. $ Passing through polar coordinates, we have 
\begin{align*}
\widetilde{\nu}(\R^d)& = \il_{\R^d} \nu(R(1+|h|))\d h =|\mathbb{S}^{d-1}| \int_{0}^{\infty} \nu(R(1+r)) r^{d-1} \d r\\
&= |\mathbb{S}^{d-1}|R^{-1} \int_{R}^{\infty} \nu(r) \Big(\frac{r}{R}-1\Big)^{d-1} \d r\leq  |\mathbb{S}^{d-1}|R^{-d} \int_{R}^{\infty} \nu(r) r^{d-1} \d r\\
&=  R^{-d} \int_{|h|\geq R} (1\land |h|^2)  \nu(h) \d h \leq R^{-d} \int_{\R^d} (1\land |h|^2) \nu(h) \d h< \infty\, .
\end{align*}
This proves $(i)$ and hence $L^2(\R^d,\widetilde{\nu}(h)\d h) \subset L^1(\R^d,\widetilde{\nu}(h)\d h)$. Let $u\in \VnuOm \subset L^2(\Omega)$. Then 
\begin{align*}
\int_{\Omega}u^2(x)\d x+\iint\limits_{\Omega\Omega^c} &(u(x)-u(y))^2\nu(x-y)\d y\d x
\\
&=  \widetilde{\nu}(\Omega^c)^{-1} \iint\limits_{\Omega\Omega^c} u^2(x)\widetilde{\nu}(y)\d y\d x+\iint\limits_{\Omega\Omega^c} (u(x)-u(y))^2\nu(x-y)\d y\d x\\
&\geq  (1\land  \widetilde{\nu}(\Omega^c)^{-1})(1\land c^{-1}) \iint\limits_{\Omega\Omega^c}\Big[ u^2(x)+ (u(x)-u(y))^2 \Big] \widetilde{\nu}(y)\d y\d x\\
&\geq  (1\land  \widetilde{\nu}(\Omega^c)^{-1})(1\land c^{-1})\frac{|\Omega|}{2} \int\limits_{\Omega^c}u^2(y)  \widetilde{\nu}(y)\d y\, .
\end{align*}
Moreover, note that for an appropriate constant $C>0$ we have  $C^{-1}\|u\|_{L^2(\Omega)} \leq\|u\|_{L^2(\Omega, \, \widetilde{\nu}(h)\d h)} \leq C \|u\|_{L^2(\Omega)} $ since $R\leq R(1+|h|)\leq R(1+R)$ for all $ h\in \Omega$. This together with the previous estimate shows $u\in L^2(\R^d, \widetilde{\nu})$. Therefore, the proof of $(iii)$ is complete. Obviously, we also have  $\vertiii{u}_{\VnuOm}\leq C\vertiii{u}^*_{\VnuOm}$.  The reverse inequality  is an immediate consequence of the above estimates, thereby proving the equivalence of the two norms under consideration. Part $(iv)$ is proved.

\end{proof}

\medskip

\begin{proposition}
Let $\alpha_0\in (0,2)$ be as in \eqref{eq:consequence-integrability}.  The quadratic forms $(\mathcal{E}^{\alpha}_{\Omega}(\cdot, \cdot), H_{\nu^\alpha} (\Omega)) $ and $\big(\mathcal{E}^{\alpha}(\cdot, \cdot), V_{\nu^\alpha}(\Omega|\mathbb{R}^d)\cap L^2(\mathbb{R}^d)\big)$  are well defined for every $\alpha\in (\alpha_0, 2)$. 
\end{proposition}

\begin{proof}
	Let $\alpha\in (\alpha_0, 2)$ . Let $u \in H_{\nu_\alpha}(\Omega)$.  By the assumption \eqref{eq:elliptic-condition} and relation \eqref{eq:consequence-integrability} we have 
	
	\begin{align*}
	\mathcal{E}^\alpha_{\Omega}(u,u) &= \hspace*{-2ex}\iil_{\Omega\Omega \cap \{|x-y|\leq 1\}} (u(x)-u(y))^2 J^\alpha(x,y)\d x\d y+ \hspace*{-2ex} \iil_{\Omega\Omega \cap \{|x-y|>1\}} (u(x)-u(y))^2 J^\alpha(x,y)\d x\d y\\
	&\leq   \Lambda \hspace*{-2ex}\iil_{\Omega\Omega \cap \{|x-y|\leq 1\}} (u(x)-u(y))^2 \nu^\alpha(x-y)\d x\d y+ 4 \il_{\Omega} u^2(x)\d x\il_{|x-y|>1} J^\alpha(x,y)\d y\\
	&\leq   \Lambda \iil_{\Omega\Omega } (u(x)-u(y))^2 \nu^\alpha(x-y)\d x\d y+ 4 \kappa_0 \il_{\Omega} u^2(x)d x\\
	&\leq (\Lambda+4\kappa_0)\|u\|^2_{H_{\nu^\alpha}(\Omega)}<\infty\, .
	\end{align*}
Now if $u \in V_{\nu_\alpha}(\Omega|\mathbb{R}^d) $ then, from the above  we  deduce  $\mathcal{E}^\alpha_{\Omega}(u,u)<\infty$. By the same argument we obtain 
	\begin{align*}
	&\iil_{\Omega\Omega^c } (u(x)-u(y))^2 J^\alpha(x-y)\d x\d y\\&\leq \Lambda \hspace*{-2ex}\iil_{\Omega\Omega^c \cap \{|x-y|\leq 1\}} (u(x)-u(y))^2 \nu^\alpha(x-y)\d x\d y+ 2\hspace*{-2ex} \iil_{\Omega\Omega^c \cap \{|x-y|>1\}} (u^2(x)+u^2(y)J^\alpha(x,y)\d x\d y\\
	&\leq   \Lambda \iil_{\Omega\Omega^c\cap \{|x-y|\leq 1\}} (u(x)-u(y))^2 \nu^\alpha(x-y)\d x\d y+ 2 \kappa_0 \il_{\Omega} u^2(x)d x+ 2 \kappa_0 \il_{\Omega^c} u^2(x)d x\\
	 &\leq  \Lambda \iil_{\Omega\Omega^c} (u(x)-u(y))^2 \nu^\alpha(x-y)\d x\d y+ 2 \kappa_0 \il_{\mathbb{R}^d} u^2(x)dx<\infty\, .
	\end{align*}
	Finally, we obtain 
	\begin{align*}
	\mathcal{E}^\alpha(u,u) &= \mathcal{E}^\alpha_{\Omega}(u,u) + 2\iil_{\Omega\Omega^c } (u(x)-u(y))^2 J^\alpha(x-y)\d x\d y<\infty\, . 
	\end{align*}

\end{proof}

\medskip

\begin{definition}[cf. \cite{Adams}] In what follows, a domain $D \subset \mathbb{R}^d$ is called an extension domain if there exists a linear operator $E: H^1(D)\to H^1(\mathbb{R}^d)$ and a constant $C: = C(D, d)$ depending only on the domain $D$ and the dimension $d$ such that for all $u \in H^1(D)$
	\begin{align*}
	Eu|_{D} = u \qquad\hbox{and} \qquad \|Eu\|_{H^1(\mathbb{R}^d)}\leq C \|u\|_{H^1(D)}.
	\end{align*}
\end{definition}

The next lemma shows that the nonlocal quadratic forms under consideration are continuous on $H^1(D)$.

\begin{lemma}\label{lem-cont-qua-form}
	
	Assume $ D \subset \mathbb{R}^d$ be an extension domain.  Assume $J^\alpha $ satisfies \eqref{eq:elliptic-condition} and \eqref{eq:integrability-condition} and let   $\alpha_0\in  (0,2)$be as in \eqref{eq:consequence-integrability}.  Then, there exists a constant $C:= C(D, \Lambda, d, \alpha_0)$ such that for every $u \in H^1(D)$ and every $\alpha\in  (\alpha_0,2)$

	\begin{align*}
	\mathcal{E}^{\alpha}_D(u,u) \leq C\| u\|^2_{H^1(D)}.
	\end{align*}
	
\end{lemma}

\begin{proof}
	Firstly,  from the symmetry of  $J^\alpha(x,y)$ and \eqref{eq:consequence-integrability} we have the following estimates
	\begin{eqnarray*}
		\iil_{D\times  D \cap \{|x-y|\geq 1\}} (u(x)-u(y))^2 J^\alpha(x,y)\d x \, \d y& \leq & 2\il_{D}u^2(x) \d x \il_{ |x-y|\geq 1} J^\alpha(x,y) \, \d y\\
		&\leq& 2\kappa_0\|u\|^2_{L^2(D)}.
	\end{eqnarray*}

	Now, let $\overline{u}\in H^{1}(\mathbb{R}^d)$ be an extension of $u$ then upon the estimate 
	$\|\overline{u}(\cdot+ h)-\overline{u}\|_{L^2(\mathbb{R}^d)} \leq  
	|h| \|\nabla\overline{u}\|_{L^2(\mathbb{R}^d)}$ (which can be established through density of smooth functions with compact support in   $H^{1}(\mathbb{R}^d)$) we have

	\begin{align*}
	\iil_{D\times D\cap \{|x-y|\leq  1\}} &\frac{(u(x)-u(y))^2}{|x-y|^{2}} \rho_{2-\alpha}(x-y) \d x \, \d y 
=  \iil_{D\times D\cap \{|x-y|\leq  1\}} \frac{(\overline{u}(x)-\overline{u}y))^2}{|x-y|^{2}}  \rho_{2-\alpha}(x-y) \d x \, \d y \\
	&\leq  \il_{ |h|\leq  1}   \rho_{2-\alpha}(h) \frac{ \d h }{|h|^{2}} \il_{\mathbb{R}^d} (\overline{u}(x+h)-\overline{u}(x))^2 \d x \\
	&\leq  \|\nabla\overline{u}\|_{L^2(\mathbb{R}^d)} \il_{|h|\leq 1} \rho_{2-\alpha}(h) \d h 
	\leq  C \|u\|^2_{H^{1}(D)}. 
	\end{align*}
	Precisely, we have
	\begin{align*}
	\iil_{D\times  D \cap \{|x-y|\leq 1\}} (u(x)-u(y))^2 |x-y|^{-2} \rho_{2-\alpha}(x-y)\d x \, \d y
	&\leq C\| u\|^2_{H^1(D)}.
	\end{align*}
	Combining the above estimates along with  the condition \eqref{eq:elliptic-condition} we get,
	$$ \mathcal{E}^{\alpha}_D(u,u) \leq C\| u\|^2_{H^1(D)}. $$
	
\end{proof}

\medskip

\begin{proposition} Let $\nu$ be as above. The function spaces  $\big(V_\nu (\Omega|\mathbb{R}^d), \|\cdot\|_{V_\nu (\Omega|\mathbb{R}^d)}\big)$ and $\big(H_\nu (\Omega), \|\cdot\|_{H_{\nu} (\Omega)}\big)$ are separable Hilbert spaces. If $\nu$ has full support in $\R^d$, i.e. if $\nu>0$ a.e on $\R^d$, then the same is true for the space $\big(V_\nu (\Omega|\mathbb{R}^d), \vertiii{\cdot}_{V_\nu (\Omega|\mathbb{R}^d)}\big)$.
\end{proposition}  

\noindent For the proof we follow ideas from \cite{FKV15, DRV17}. 
\begin{proof}
It is not difficult to check that, $ \|\cdot\|_{V_\nu (\Omega|\mathbb{R}^d)}$ and $ \|\cdot\|_{H_{\nu} (\Omega)}$ are norms on  $V_\nu (\Omega|\mathbb{R}^d)$ and $H_{\nu} (\Omega)$ respectively. 
Now, if  $\vertiii{u}_{V_\nu (\Omega|\mathbb{R}^d)} = 0,$ then, $u=0$ a.e on $\Omega$ and since  $[u]^2_{V_\nu (\Omega|\mathbb{R}^d)}=0$ with $\nu(h)>0$ a.e we have $u(y)=u(x) =0$ for almost all $(x,y)\in \Omega\times \R^d$. That, is $u=0$ a.e on $\R^d$ and this enables $\vertiii{\cdot}_{V_\nu (\Omega|\mathbb{R}^d)}$  to be a  norm on $V_\nu (\Omega|\mathbb{R}^d). $ 

\medskip

Now, let $(u_n)_n$ be a Cauchy sequence in $\big(V_\nu (\Omega|\mathbb{R}^d), \|\cdot\|_{V_\nu (\Omega|\mathbb{R}^d)}\big)$. It converges to some $u$ in the topology of $L^2(\mathbb{R}^d)$ and pointwise almost everywhere in $\mathbb{R}^d$ up to a subsequence $(u_{k_n})_n$. Fix $n$ large enough, the Fatou lemma implies 
\begin{align*}
[u_{k_n}-u]^2_{V_\nu (\Omega|\mathbb{R}^d)} \leq  \liminf_{\ell\to \infty}  \iil_{(\Omega^c\times \Omega^c)^c} \big([u_{k_n}-u_{k_\ell}](x)-([u_{k_n}-u_{k_\ell}](y) \big)^2 \, \nu (x-y) \mathrm{d}x \, \mathrm{d}y \,
\end{align*}
Since $(u_{k_n})_n$ is a Cauchy sequence, the right hand side is finite for any $n$  and tends to $0$ as $n\to \infty$. This implies $u\in V_\nu (\Omega|\mathbb{R}^d)$ and $[u_{k_n}-u]^2_{V_\nu (\Omega|\mathbb{R}^d)} \to 0$ as $n\to \infty$. Finally, $u_n\to u$ in $V_\nu (\Omega|\mathbb{R}^d)$.  Furthermore, the map $\mathcal{I}: V_\nu (\Omega|\mathbb{R}^d)\to L^2(\R^d) \times L^2(\Omega\times \R^d)$ with 
\begin{align*}
\mathcal{I}(u) = \Big(u(x), (u(x)-u(y))\sqrt{\nu(x-y)}\Big)
\end{align*}
is an isometry. Hence  from its Hilbert structure, the space  $\big(V_\nu (\Omega|\mathbb{R}^d), \|\cdot\|_{V_\nu (\Omega|\mathbb{R}^d)}\big)$, which can be identified with $\mathcal{I}\Big(V_\nu (\Omega|\mathbb{R}^d)\Big)$, is  separable as a closed subspace of the separable space  $ L^2(\R^d) \times L^2(\Omega\times \R^d)$. 
Analogously, one  shows that,  $\big(H_\nu (\Omega), \|\cdot\|_{H_{\nu} (\Omega)}\big)$ is a separable Hilbert space. 

\medskip

It remains to prove that $\big(V_\nu (\Omega|\mathbb{R}^d), \vertiii{\cdot}_{V_\nu (\Omega|\mathbb{R}^d)}\big)$ is a separable Hilbert space. Here we assume that $\nu$ has full support on $\R^d$. Without loss of generality we assume $\nu(h)>0$ for every $h\in \R^d$.   
Assume that $(u_n)_n$ is  a Cauchy sequence in $\big(V_\nu (\Omega|\mathbb{R}^d), \vertiii{\cdot}_{V_\nu (\Omega|\mathbb{R}^d)}\big)$. Then there exist a subsequence $(u_{k_n})_n$ of $(u_{n})_n$, a function $u$ in $L^2(\Omega)$, a function $U \in L^{2}(\Omega\times \mathbb{R}^d)$, and null sets $N\subset \mathbb{R}^d$  and $\mathcal{R}\subset \Omega\times \mathbb{R}^d$ such that 
\begin{itemize}
	\item[-] $(u_{k_n})_n$ converges to $u$ in  $L^2(\Omega)$\,,
	\item[-] $(u_{k_n})_n$ converges to $u$ pointwise on $\Omega\setminus N$\,,
	\item[-] $(U_{k_n})_n$ converges to $U$ in  $L^2(\Omega \times \R^d)$\,,
	\item[-] $(U_{k_n})_n$ converges to $U$ pointwise on $(\Omega\times \R^d)\setminus \mathcal{R}$\,,
\end{itemize}
where $U_n(x,y) = (u_n(x)-u_n(y))\sqrt{\nu(x-y)}$. Let $(x,y)\in (\Omega \times \R^d)\setminus \mathcal{R'} $ with $x\neq y$ where $ \mathcal{R'}= \mathcal{R} \cup (N\times \emptyset)$. Then,  as $n\to \infty$ we have  

\begin{align*}
u_{n_k}(y)= u_{n_k}(x)  - U_{n_k}(x,y)/\sqrt{\nu(x-y)} \to  u(x)  - U(x,y)/\sqrt{\nu(x-y)}
\end{align*}

\medskip

Finally, $U(x,y) = (u(x)-u(y))\sqrt{\nu(x-y)} \in L^2(\Omega \times \R^d)$ so that $u \in \VnuOm $. We easily conclude $\vertiii{u_n-u}_{V_\nu (\Omega|\mathbb{R}^d)}\to 0$ as $n\to \infty$, which proves completeness. Let us mention that, alternatively, one could apply the equivalence of the norms $\vertiii{\cdot}_{V_\nu (\Omega|\mathbb{R}^d)}$ and $\vertiii{\cdot}^*_{V_\nu (\Omega|\mathbb{R}^d)}$, cf. \autoref{prop:natural-norm-on-V} $(iv)$. This would allow to establish completeness along the lines of the proof of completeness in the first case. The separability of the space $\big(V_\nu (\Omega|\mathbb{R}^d), \vertiii{\cdot}_{V_\nu (\Omega|\mathbb{R}^d)}\big)$ can be shown as in the case above. 
\end{proof}

\begin{remark}
Let us define spaces of functions that vanish on the complement of $\Omega$. Set
\begin{align}\label{eq:VnuOm-vanish}
V^\Omega_{\nu}(\Omega|\mathbb{R}^d)= \{ u\in V_{\nu}(\Omega|\mathbb{R}^d)~| ~u=0~~\text{a.e. on } \mathbb{R}^d\setminus \Omega\}\,.
\end{align} 
As a direct direct consequence of \autoref{prop:natural-norm-on-V}, the space $\big(V^\Omega_{\nu}(\Omega|\mathbb{R}^d), \|\cdot \|_{V_{\nu}(\Omega|\mathbb{R}^d)} \big)$ is a separable Hilbert space, too. Both norms $\|\cdot \|_{V_{\nu}(\Omega|\mathbb{R}^d)} $ and $\vertiii{\cdot}_{V_{\nu}(\Omega|\mathbb{R}^d)} $ coincide on $V^\Omega_{\nu}(\Omega|\mathbb{R}^d)$.
\end{remark}

\medskip

Finally, we are in the position to prove our first main result, \autoref{thm:density}.

\begin{proof}[Proof of \autoref{thm:density}]
	Assume $u \in V_{\nu} (\Omega|\mathbb{R}^d)$. We prove that there is a sequence $(u_n)$ of functions in $C^\infty_c(\R^d)$ such that $[u-u_n]_{\VnuOm}$ converges to $0$ as $n\to\infty$. This implies
	\begin{align}\label{eq:density-conv}
	\|u_n-u\|_{V_{\nu}(\Omega|\mathbb{R}^d)} \longrightarrow 0 \text{ as } n\to\infty \,,
	\end{align} 
	since the convergence $\|u_n - u\|_{L^2(\R^d)}$ follows by standard arguments. 
	Obviously, the convergence $\vertiii{u_n-u}_{V_{\nu}(\Omega|\mathbb{R}^d)} \to 0$ follows from \eqref{eq:density-conv}.   Note that the sequence $(u_n)$ is constructed by translation and convolution of the function $u$ with a mollifier. 

\medskip

	\noindent \textbf{Step 1:} Let $x_0\in\partial\Omega$. Since $\partial\Omega$ Lipschitz, there exists $r>0$ and a Lipschitz  function $\gamma:\R^{d-1}\to\R$ with Lipschitz constant $k>0$, such that (upon relabeling the coordinates) 
	
	\begin{align*}
	\Omega\cap B_r(x_0) &= \{ x\in B_r(x_0)| x_d > \gamma(x_1,...,x_{d-1})\}.\\ 
	\end{align*}
	Set $x= (x_1,...,x_{d-1},x_d)=(x',x_d)$.  For sake of convenience, we choose $r>0$ so small such that $|\Omega\cap B^c_r(x_0)|>0$. For $x\in B_{r/2}(x_0)$, $\tau >1+k$  and $0<\eps <\frac{r}{2(1+\tau)}$  we define the shifted point 
	\[
	x_\eps =
	x + \tau \eps e_d\,.     
	\]
	We define $u_\eps(x) = u(x_\eps)= u(x+\tau \eps e_d)$ and 
	\[ v_\eps = \eta_\eps \ast u_\eps \] 
	where $\eta_\eps$ is a smooth mollifier having support in $B_\eps(0)$.

	\medskip
	
	\noindent \textbf{Step 2:} Let us assume $\supp u\Subset B_{r/4}(x_0)$. In this case $v_\eps \in C^\infty_c (B_r(x_0))$. The aim of this step is to prove 
	\[ 
	[v_\eps - u]_{\VnuOm} \longrightarrow 0 \quad \text{ as } \eps \to 0 \,. 
	\]
	Due to the nonlocal nature of the seminorm, this step turns out to be rather challenging. We begin with a geometric observation.

	\begin{lemma}\label{lem:guy-geo}
	Let $z\in B_1(0)$. Let $\Omega^z_\eps = \Omega+ \eps(\tau e_d-z)$. Then $\Omega^z_\epsilon\cap B_{r/2}(x_0) \subset \Omega\cap B_{r}(x_0)$. 
    \end{lemma}
 \begin{proof}
	For $h \in \Omega^z_\epsilon\cap B_{r/2}(x_0)$, let us write $ h= t+\eps \tau e_d-\eps z $ with $t \in \Omega$. Then $t\in B_{r/2}(x_0) $, $h' = t'-\eps z'$ and $h_d= t_d+ \eps(\tau -z_d)$. Since $\gamma$ is Lipschitz with Lipschitz constant $k<\tau-1$ and $t\in \Omega\cap B_{r/2}(x_0) = \{ x\in B_{r/2}(x_0)| x_d > \gamma(x')\}$ we obtain
	\begin{align*}
	\gamma(h')&\leq \gamma(t')+ |\gamma(h')-\gamma(t')|  <t_d+ \eps k|z'|\\
	& <t_d+ \eps k< t_d+ \eps(\tau -z_d)=  h_d.
	\end{align*}
	Thus, $h \in B_r(x_0)$ and $h_d>\gamma(h')$. We have shown $h\in \Omega \cap B_r(x_0)$ as desired. 
\end{proof}

The main technical tool of the argument below is the Vitali convergence theorem, see \cite[Chapter 3]{Alt16} or  \cite[Corollary 4.5.5.]{bogachev2007volumeI}.  Since $u$ belongs to the space $\VnuOm$, for every $\delta > 0$ there is $\eta > 0$ such that for all sets $E \subset \Omega$, $F \subset \R^d$ with $|E \times F| < \eta$ we know
\begin{align}\label{eq:equi-int-u}
\iil_{E F} \big(u(x)-u(y)\big)^2 \nu(x-y) \d y \d x &< \delta \text{ and } \iil_{E F}  u^2(y) \d y \d x < \delta \,.
\end{align}
The second estimate uses the fact that $\iint_{\Omega \R^d}  u^2(y) \d y \d x$ is finite because $u$ has compact support. As a consequence of \eqref{eq:equi-int-u}, we derive the following lemma. 

\begin{lemma}\label{lem:equi-int-u-eps}
For every $\delta > 0$ there is $\eta > 0$ such that for all sets $E \subset \Omega$, $F \subset \R^d$ with $|E \times F| < \eta$ 
\begin{align}\label{eq:equi-int-u_eps-z}
\sup\limits_{z \in B_1(0)} \sup\limits_{\eps > 0} \iil_{E F} \big(u^z_\eps(x)-u^z_\eps(y)\big)^2 \nu(x-y) \d y \d x < \delta \,,
\end{align}
where $ u^z_\eps(\xi) = u_\eps(\xi-\eps z) = u(\xi+\eps \tau e_d-\eps z)$.
\end{lemma}

\begin{proof}
Let $\delta > 0$. Choose $\eta > 0$ as in \eqref{eq:equi-int-u}. Let $\eps > 0$, $z \in B_1(0)$. Let $E \subset \Omega$, $F \subset \R^d$ be sets with with $|E \times F| < \eta$. Then
\begin{align}
\iil_{E F} \big(u^z_\eps(x)-u^z_\eps(y)\big)^2 \nu(x-y) \d y \d x = \iil_{E^z_\eps F^z_\eps} \big(u(x)-u(y)\big)^2 \nu(x-y) \d y \d x \,,
\end{align}
where $E^z_\eps = E + \eps(\tau e_d-z)$ and $F^z_\eps$ defined analogously. We decompose $E^z_\eps$ as follows $E^z_\eps = E^z_\eps \cap B_{r/2}(x_0) \cup E^z_\eps \cap B^c_{r/2}(x_0)$. Note 
\[ E^z_\eps \cap B_{r/2}(x_0) \subset \Omega^z_\eps \cap B_{r/2}(x_0) \subset \Omega \cap B_{r/2}(x_0) \,,\]
where we apply \autoref{lem:guy-geo}. We directly conclude
\begin{align}\label{eq:equi-one}
\iil_{E^z_\eps   F^z_\eps} \mathbbm{1}_{B_{r/2}(x_0)}(x) \big(u(y)-u(x)\big)^2 \nu(x-y) \d y \d x \leq \delta \,.
\end{align}
With regard to the remaining term note
\begin{align}\label{eq:equi-two}
\begin{split}
\iil_{E^z_\eps   F^z_\eps}& \mathbbm{1}_{B^c_{r/2}(x_0)}(x) \big(u(x)-u(y)\big)^2 \nu(x-y) \d y \d x \\ 
&= 
\iil_{E^z_\eps   F^z_\eps} \mathbbm{1}_{B^c_{r/2}(x_0)}(x) \mathbbm{1}_{B_{r/4}(x_0)}(y) u^2(y) \nu(x-y) \d y \d x\\
&\leq c(r, \nu) \iil_{E^z_\eps   F^z_\eps} \mathbbm{1}_{B^c_{r/2}(x_0)}(x) \mathbbm{1}_{B_{r/4}(x_0)}(y) u^2(y) \d y \d x \leq c \iil_{E^z_\eps   F^z_\eps}  u^2(y) \d y \d x\\
&= c \iil_{E F^z_\eps}  u^2(y) \d y \d x\leq c \delta \,.
\end{split}
\end{align}
The positive constant $c(r,\nu)$ depends on $r$ and on the shape of $\rho$. Summation over \eqref{eq:equi-one} and \eqref{eq:equi-two} completes the proof after redefining $\delta$ accordingly. 
\end{proof}

The next lemma shows the tightness of $u^z_\eps(x)-u^z_\eps(y)$ uniformly for $z\in B_1(0)$ and $\eps >0$.

\begin{lemma}\label{lem:tigh-u-eps-z}
For every $\delta>0$ there exists $E_\delta \subset \Omega$ and $F_\delta \subset \mathbb{R}^d$ such that $|E_\delta \times F_\delta |<\infty$  and 
\begin{align}\label{eq:tight-u_eps}
\sup\limits_{z \in B_1(0)} \sup\limits_{\eps > 0} \iil_{(\Omega \times \mathbb{R}^d) \setminus ( E_\delta \times F_\delta)} \big(u^z_\eps(x)-u^z_\eps(y)\big)^2 \nu(x-y) \d y \d x < \delta.
\end{align}
\end{lemma}

\begin{proof}
Fix $\eps>0$ and $z\in B_1(0)$. Let $ \bar{R}= \sup\limits_{\xi \in \Omega} |\xi-x_0|$ which is finite since $\Omega $ is bounded.  Note that $\supp u^z_\eps \subset B_{r/2}(x_0)$. Choose $R>0$ so large such that $[B^c_{R}(x_0)]_{\eps}^z= B^c_{R}(x_0) +\eps (\tau e_d+z)\subset B^c_{R/2}(x_0) $  and  $|x-y|\geq R/2-\bar{R}$  for $x\in B^c_{R/2}(x_0) $ and $y \in \Omega$. Thus,
\begin{align*}
\iil_{(\Omega \times \mathbb{R}^d) \setminus ( \Omega \times B_R(x_0))} & \big(u^z_\eps(x)-u^z_\eps(y)\big)^2  \nu(x-y) \d y \d x
= \iil_{\Omega B^c_R(x_0)} \big(u^z_\eps(x)\big)^2 \nu(x-y) \d y \d x\\
&= \il_{\Omega^z_\eps \cap B_{r/2}(x_0) } u^2(x)\d x \il_{ [B^c_{R}(x_0)]_{\eps}^z}  \nu(x-y) \d y \leq \il_{\Omega} u^2(x)\d x  \il_{ B^c_{R/2-\bar{R}}(x)}  \nu(x-y) \d y\\ 
&= \|u\|^2_{L^2(\Omega)}  \il_{ B^c_{R/2-\bar{R}}(0)} \nu(h) \d h. 
\end{align*}
The  desired result follows by taking   $E_\delta = \Omega$ and $F_\delta =B_R(x_0)$ with  $R>0$ large enough such that $ \il_{ B^c_{R/2-\bar{R}}(0)} \nu(h) \d h<\delta \|u\|^{-2}_{L^2(\Omega)}$ . 
\end{proof}
\begin{lemma} There exists a constant $C(\Omega,r, \nu)$ depending only on $\Omega,r$ and $\nu$ such that, for all $z\in B_1(0)$ and all $\eps>0$
\begin{align}\label{eq:estimate-seminorm}
 [u^z_\eps]^2_{\VnuOm}\leq C(\Omega,r, \nu) [ u]^2_{\VnuOm}.
\end{align} 
\end{lemma}
		
\begin{proof}
Note that, $|x-y|\geq r/4$ for $x\in B^c_{r/2}(x_0) $ and $y \in B_{r/4}(x_0)$ and   there is  $c_r(\Omega, \nu)>0$ such that $ \nu(x-y) >c_r(\Omega,\nu)$ for all $x\in \Omega$ and all $y \in B_{r/4}(x_0)$ since $\Omega$ is bounded.  Let us chose $ C=C(\Omega,r, \nu)$ not less than 
\[ 1+c_r^{-1}(\Omega,\nu)|\Omega\cap B_r^c(x_0)|^{-1}\int\limits_{ B^c_{r/4}(0)} \nu(h)   \d h.\]
Therefore, for each $z\in B_1(0)$ and each $\eps>0 $ we have 
\begin{align*}
 \iint\limits_{ \Omega^z_\epsilon\cap B^c_{r/2}(x_0)\times \R^d} \left(  u(x) - u(y) \right)^2 \nu(x-y)  \d y \d x
&= \int\limits_{ B_{r/4}(x_0) } u^2(y)\d y \int\limits_{ \Omega^z_\epsilon\cap B^c_{r/2}(x_0)}  \nu(x-y)   \d x\\
&\leq \int\limits_{ B_{r/4}(x_0) } u^2(y)\d y \int\limits_{ B^c_{r/4}(y)}  \nu(x-y)   \d x\\
&\leq  C\hspace{-3ex}  \int\limits_{ B_{r/4}(x_0) } u^2(y)\d y \int\limits_{ \Omega \cap B^c_{r}(x_0)}  \nu(x-y)   \d x\\
&= C \hspace{-3ex}\iint\limits_{ \Omega \cap B^c_{r}(x_0)\times\R^d }(u(x)-u(y))^2  \nu(x-y)  \d y \d x.\
\end{align*}
Using a change of variables, this  and \autoref{lem:guy-geo}, we have 
\begin{align*}
&[ u^z_\eps]^2_{\VnuOm} = \iint\limits_{ \Omega\R^d} \left(  u^z_\eps(x) - u^z_\eps(y) \right)^2 \nu(x-y)  \d y \d x
= \iint\limits_{ \Omega^z_\epsilon\R^d} \left(  u(x) - u(y) \right)^2\nu (x-y)  \d y \d x\\
&= \iint\limits_{ \Omega^z_\epsilon\cap B_{r/2}(x_0)\times \R^d}\hspace*{-2ex} \left(  u(x) - u(y) \right)^2 \nu(x-y)  \d y \d x
+ \hspace{-3ex} \iint\limits_{ \Omega^z_\epsilon\cap B^c_{r/2}(x_0)\times \R^d} \left(  u(x) - u(y) \right)^2 \nu(x-y)  \d y \d x\\
&\leq C\hspace{-3ex} \iint\limits_{ \Omega\cap B_{r}(x_0)\times \R^d}\hspace*{-2ex} \left(  u(x) - u(y) \right)^2 \nu(x-y)  \d y \d x
+C\hspace{-3ex} \iint\limits_{ \Omega \cap B^c_{r}(x_0)\times\R^d }(u(x)-u(y))^2  \nu(x-y)  \d y \d x\\
&= C [u]^2_{\VnuOm}.
\end{align*}
\end{proof}	

We are now in position to prove the main result of this step. By Jensen's inequality, we get the following 

\begin{align*}
	&[v_\eps - u]^2_{\VnuOm} = \iil_{\Omega\,\R^d} ((v_\eps(x)-v_\eps(y)) -(u(x)-u(y)))^2
	 \nu(x-y)\d y \d x\\
	&= \iil_{\Omega\,\R^d} \Big(\il_{\R^d} ((u_\eps(x-z)-u_\eps(y-z))\eta_\eps(z) \d z 
	-(u(x)-u(y))\Big)^2\nu(x-y) \d y \d x \\
	&=  \iil_{\Omega\,\R^d} \Big(\il_{B_1(0)} ((u_\eps(x-\eps z) - u_\eps(y-\eps z) ) 
	-(u(x)-u(y)))  \eta(z) \d z \Big)^2 \nu(x-y) \d y \d x \\
	&\leq \iint\limits_{\Omega \R^d} \int\limits_{B_1(0)} \big( (u_\eps(x-\eps z) - u_\eps(y-\eps z) ) 
	-(u(x)-u(y))\big)^2\nu(x-y) \eta(z) \d z \d y \d x \\
	&=\int\limits_{B_1(0)}  \eta(z) \iint\limits_{\Omega \R^d} \big( ( u_\eps(x-\eps z) - u_\eps(y-\eps z) ) 
	-(u(x)-u(y))\big)^2\nu(x-y)  \d y \d x\,  \d z\,\\
	&= \int\limits_{B_1(0)}  [u^z_\eps-u]^2_{\VnuOm}\eta(z) \d z\,.
	\end{align*}
\noindent For each $z \in B_1(0)$ the family of functions 
$(x,y)\mapsto \left( ( u^z_\eps(x) - u^z_\eps(y) ) -(u(x)-u(y))\right)^2\nu(x-y) $ with $(x,y) \in \Omega\times\mathbb{R}^d$, $\eps>0$  is equiintegrable (by \autoref{lem:equi-int-u-eps}),   is tight (by \autoref{lem:tigh-u-eps-z}) and converges to $0$ a.e on $\Omega\times\mathbb{R}^d$. Thus for fixed $z\in B_1(0)$ the Vitali's convergence theorem gives 
\begin{align*}
	\iil_{\Omega\,\R^d} \left(( u_\eps(x-\eps z) - u_\eps(y-\eps z) ) 
	-(u(x)-u(y))\right)^2\nu(x-y)  \ \d y \ \d x\overset{\eps \to 0}{\longrightarrow} 0\,.
\end{align*}	
	That is, $[u^z_\eps-u]^2_{\VnuOm} \to 0, $ as $\eps \to 0$ for each $ z\in B_1(0)$.  Further, from estimate \eqref{eq:estimate-seminorm} the function $ z \mapsto  \eta(z)[u^z_\eps-u]^2_{\VnuOm} $ is bounded by $2C [u]_{\VnuOm}$ for all $\eps>0$ and a.e. $z\in B_1(0)$. Thus, by Lebesgue's 
	dominated convergence theorem  
	\[ \int\limits_{B_1(0)}  [u^z_\eps-u]^2_{\VnuOm}\eta(z) \d z \overset{\eps \to 0}{ \longrightarrow 0}.  
	\]
Which implies  $ [v_\eps-u]_{\VnuOm}\to 0$ as $\eps \to 0$.
	
	
	\medskip
	
	\noindent \textbf{Step 3:} 	Let $u\in \VnuOm$ be arbitrary. Let $R>0$ such that $\Omega\subset B_{R}(0)$. Let $f_R\in C_c^\infty(B_{3R}(0))$ with $f_R\leq1$ and 
	$f_R(x)=1$ for all $x\in B_{2 R}(0)$. 
	Define $u_R = f_R u$. Then $\supp(u_R)\subset B_{3R}(0)$ and $[u-u_R]_{\VnuOm} \to 0$ as $R \to \infty$. 
	
	\medskip
	
	\noindent \textbf{Step 4:} Let $x_i\in\partial\Omega$, $r_i>0$, $i=1,..,N$, such that
	\[
	\partial\Omega \subset \bigcup_{i=1}^N B_{r_i/2}(x_i), 
	\]
	where the $r_i$ are chosen small enough, such that (up to relabeling the coordinates) we can assume 
	\begin{align*}
	\Omega\cap B_{4r_i}(x_i) &= \{ x\in B_{4r_i}(x_i)| x_d > \gamma_i(x')\}\\
	\end{align*}
	for some smooth $\gamma_i:\R^{d-1}\to\R$ as in Step 1. Let $\Omega^*= \{x\in \R^d| \dist(x,\Omega)>\frac12 \min_{i=\{1,..,N\}} r_i\}$ 
	and $\Omega_0 = \{x\in \Omega| \dist(x,\Omega^c)>\frac12 \min_{i=\{1,..,N\}} r_i\}$. Then 
	\[
	\bigcup_{i=1}^N B_{r_i}(x_i) \cup \Omega^*\cup\Omega_0 = \R^d .
	\]
	Let $\{\xi_i\}_{i=0}^{N+1}$ be a smooth partition of unity subordinated to the above constructed sets. \\
	We define 
	\[ u_i = \xi_i\cdot u_R \text{ for all } i\in\{0,..,N+1\}, \] 
	and thus 
	\begin{align*}
	&\supp u_i \subset  B_{r_i}(x_i) \text{ for }i\in\{1,..N\}, \\
	&\supp u_0 \subset \Omega_0, \\
	&\supp u_{N+1} \subset \Omega^*. 
	\end{align*}
	
	\medskip
	
	\noindent \textbf{Step 5:} In this step, we use the shorthand notation $\Delta u(x;y) = u(x)-u(y) $. Let $\delta>0$ and  $i\in\{1,..,N\}$. By Step 2 there exists a sequence $v^i_\eps\in C_c^\infty(B_{4r_i}(x_i))$ such that 
	\[
	[u_i-v^i_\eps]_{\VnuOm} \longrightarrow 0 
	\]
	for $\eps\to 0$. Thus we can choose $\eps_0>0$ such that $[u_i-v^i_\eps]_{\VnuOm}< \frac{\delta}{N+2}$
	for all $i\in\{1,..,N\}$.
	
	For $i=N+1$ define $v^{N+1}_\eps = \eta_\eps \ast u_{N+1}$ and set $r=\frac14 \min_{i\in\{1,..,N\}} r_i$. 
	Choosing $\eps< r$ and since  
	$\supp u_{N+1} \subset \Omega^*$ for all $x\in\Omega$, $y\in\R^d$ and $z\in B_\eps(0)$ 
	\[
	\Delta u_{N+1}(x;y) = \Delta v_\eps^{N+1}(x-z;y-z)= 0  \quad \text{ or } \quad \bet{x-y}>r.
	\]
	Thus 
	\begin{align*}
	[v^{N+1}_\eps&-u_{N+1}]^2_{\VnuOm} = \iil_{\Omega\,\R^d} (\Delta v^{N+1}_\eps(x;y)-\Delta u^{N+1}(x;y))^2 
	\nu(x-y) \d y \d x \\
	&= \iil_{\Omega \R^d} \left(\il_{B_\eps(0)} \Delta u_{N+1}(x-z;y-z)-\Delta u_{N+1}(x;y) \eta_\eps(z) \d z\right)^2 
	\nu(x-y) \d x \d y \\
	&\leq C_r \iiint\limits_{B_1(0)\times\Omega\times\R^d} (\Delta u_{N+1}(x-\eps z;y-\eps z)-\Delta u_{N+1}(x;y))^2 
	\d y \d x \eta(z) \d z .
	\end{align*}
	By the continuity of the shift in $L^2(\R^d)$
	\[ \iil_{\Omega\,\R^d} (\Delta u_{N+1}(x-\eps z;y-\eps z)-\Delta u_{N+1}(x;y))^2 \d y \d x \longrightarrow 0. \]
	Further, for any $z\in B_1(0)$, the map 
	\[
	z \mapsto \bet{\eta(z)  \iint\limits_{\Omega\,\R^d} (\Delta u_{N+1}(x-\eps z;y-\eps z)-\Delta u_{N+1}(x;y))^2 
		\nu(x-y) \d y \d x }  
	\]
	is bounded. Thus $[v^{N+1}_\eps-u_{N+1}]_{\VnuOm}\to 0$ by dominated convergence and we find
	$\eps_0>0$, such that  $[v^{N+1}_\eps-u_{N+1}]_{\VnuOm} < \frac{\delta}{N+2}$ for all $\eps<\eps_0$.
	We define $v^{0}_\eps = \eta_\eps \ast u_{0}$. Thus for $\eps<r$ 
	\[\supp v^0_\eps\Subset \Omega.\]
	The convergence $v^0_\eps \to u_0$ follows by the same arguments as above and we find $\eps_0>0$ such 
	that $[v_\eps^0-u_0]_{\VnuOm} < \frac{\delta}{N+2}$ for all $\eps,\eps_0$. 
	
	\medskip
	
	\noindent \textbf{Step 6:} Define $v_\eps = \sum_{i=0}^{N+1} v^i_\eps\in C^\infty_c(\R^d)$. Since $u_R(x) = \sum_{i=0}^{N+1} u_i(x)$, we have 
	\begin{align*}
	[u_R-v_\eps]_{\VnuOm} &\leq \left[\sum_{i=0}^{N+1} \left(v^i_\eps - u_i\right)\right]_{\VnuOm} \\
	& \leq \sum_{i=0}^{N+1} [v^i_\eps -u_i]_{\VnuOm}\\
	& \leq (N+2) \frac{\delta}{N+2} .
	\end{align*}
	Choosing $R=\frac{1}{\eps}$ in Step 3, concludes 
	\[[u-v_\eps]_{\VnuOm}\leq [u-u_{R}]_{\VnuOm}+ [u_R-v_\eps|_{\VnuOm}
	\overset{\eps\to 0}{\longrightarrow}0.\]
	The convergence in $L^2(\R^d)$ follows from the continuity of the shift in $L^2(\R^d)$. 
\end{proof}

\pagebreak[3]

The density of $C^\infty_c(\R^d)$ has a direct consequence for the nonlocal bilinear form under consideration. Concerning the definition of $\nu^\alpha$, the reader might consult \autoref{def:nu-alpha}.

\begin{corollary}\label{cor:regular-DF} Assume $\Omega\subset\mathbb{R}^d$ is a bounded domain with Lipschitz continuous boundary. Assume $J^\alpha$ satisfies \eqref{eq:elliptic-condition} and \eqref{eq:integrability-condition}. Then the bilinear forms $(\mathcal{E}^\alpha, (V_{\nu^\alpha}(\Omega|\R^d) \cap L^2(\R^d))$ and $(\mathcal{E}_\Omega^\alpha, H_{\nu^\alpha}(\Omega))$ are regular Dirichlet forms on $L^2(\R^d)$ resp. $L^2(\Omega)$. 
\end{corollary}

Note that the bilinear form $(\mathcal{E}^A,H^1(\Omega))$ is a regular Dirichlet form on $L^2(\Omega)$, which follows from the fact that $\Omega$ is an extension domain.

\begin{corollary}\label{cor:regular-DF-classical} Assume $\Omega\subset\mathbb{R}^d$ is a bounded domain with Lipschitz continuous boundary. Assume that $\nu^\alpha$ has full support. Set $J^\alpha (x,y) = \nu^\alpha(x-y)$ and let $\widetilde{\nu^\alpha}$ be given as in \autoref{prop:natural-norm-on-V}.  Then the bilinear form $(\mathcal{E}^\alpha, V_{\nu^\alpha}(\Omega|\R^d))$ is a regular Dirichlet form on $L^2(\R^d, \widetilde{\nu^\alpha})$. In particular, if $J^\alpha$ is given by $J^\alpha_1$ as in \autoref{ex:J-guys-singular}, then the bilinear form $(\mathcal{E}^\alpha, V^{\alpha/2}(\Omega|\R^d))$ is a regular Dirichlet form on $L^2(\R^d, \frac{\d x}{1+|x|^{d+\alpha}})$.

\end{corollary}

\medskip

The next density theorem is proved in \cite[Theorem A.4]{BGPR17} and it is adapted from the main result in \cite{FSV15} for fractional Sobolev spaces. A more general result is provided by \cite[Theorem 3.3.9]{CF12}.

\begin{theorem}\label{thm:density-omega}
	Assume $\Omega$ has a continuous boundary. Let $\nu$ be a L\'evy measure. Then $C_c^\infty(\Omega)$ is dense in the space (cf. \eqref{eq:VnuOm-vanish})  $\big(V^\Omega_{\nu}(\Omega|\mathbb{R}^d), \|\cdot \|_{V_{\nu}(\Omega|\mathbb{R}^d)} \big)$ .

%
  \end{theorem}
  \medskip
  
\noindent A counterpart of \autoref{cor:regular-DF}  is given by the following. 

\begin{corollary}\label{cor:regular-DF-omega} Assume $\Omega\subset\mathbb{R}^d$ is a bounded domain with continuous boundary. Assume $J^\alpha$ satisfies \eqref{eq:elliptic-condition} and \eqref{eq:integrability-condition}. The bilinear forms  $(\mathcal{E}^\alpha, V^\Omega_{\nu^\alpha}(\Omega|\R^d))$ and $\big(\mathcal{E}_\Omega^\alpha, \overline{C_c^\infty(\Omega)}^{ H_{\nu^\alpha}(\Omega)}\big)$ are regular Dirichlet forms on $L^2(\Omega)$. 
\end{corollary}

\noindent Note that $(\mathcal{E}^A,H_0^1(\Omega))$ is a regular Dirichlet forms, too. This result holds true without any assumption on the regularity of $\partial \Omega$.

\section{Proof of \autoref{thm:Mosco-convergence}}\label{sec:main-proof}

The aim of this section is to provide the proof of \autoref{thm:Mosco-convergence}. Let us begin with a simple but important observation.  
 
\begin{proposition}\label{prop:elliptic-matrix}
	Under condition \eqref{eq:elliptic-condition} and \eqref{eq:integrability-condition}, the symmetric matrix $A$ defined as in  \eqref{eq:coef-matrix} has bounded coefficients and satisfies the ellipticity condition. Precisely, we have
	\begin{align*}
	d^{-1}\Lambda^{-1}|\xi|^2\leq \langle A(x) \xi, \xi \rangle \leq d^{-1}\Lambda |\xi|^2, \quad\text{ for every } x, \xi \in \R^d \,.
	\end{align*}
\end{proposition}

\begin{proof}
	Let $x, \xi \in \R^d$ and $|h|\leq 1$. Then Condition \eqref{eq:elliptic-condition} implies that 
	\begin{align*}
	\Lambda^{-1} \nu^\alpha(h)[\xi\cdot h]^2  \leq J^\alpha(x,x+h)[\xi\cdot h]^2 \leq \Lambda \nu^\alpha(h)[\xi\cdot h]^2\quad\text{for every }\quad \xi \in \R^d
	\end{align*}
	Note that,  by definition of the matrix $A$
	\begin{align*}\lim_{\alpha\to 2^-}\int\limits_{|h|\leq 1} J^\alpha(x,x+h)[\xi\cdot h]^2\d h= \langle A(x) \xi, \xi \rangle \,.
	\end{align*}

	\noindent On the other hand, by rotationally invariance of the Lebesgue measure, we have
	
	\begin{align*}
	&\lim_{\alpha\to 2^-}\int\limits_{|h|\leq 1}\nu^\alpha(h)[\xi\cdot h]^2\d h = \lim_{\alpha\to 2^-}\int\limits_{|h|\leq 1}\nu^\alpha(h)\sum_{1\leq i,j\leq d}\xi_i\xi_j h_ih_j\d h\\
	&=\lim_{\alpha\to 2^-} \sum_{1\leq i\leq d}\xi_i^2 \int\limits_{|h|\leq 1} h_1^2 \nu^\alpha(h)\d h=\lim_{\alpha\to 2^-} |\xi|^2 \int\limits_{|h|\leq 1} h_1^2 \nu^\alpha(h)\d h\\	
	&=\lim_{\alpha\to 2^-} |\xi|^2 d^{-1}\int\limits_{|h|\leq 1} \sum_{1\leq i\leq d}h_i^2 \nu^\alpha(h)\d h =\lim_{\alpha\to 2^-} |\xi|^2 d^{-1}\int\limits_{|h|\leq 1}  \rho_{2-\alpha}(h)\d h\\
	&= |\xi|^2 d^{-1}\,,
	\end{align*}
	which ends the proof.
\end{proof}

Let us recall the notion of Mosco convergence  on $L^2$- spaces according to \cite[Definition 2.1.1.]{Mos94}.
\begin{definition}[Mosco-convergence]\label{def:mosco}
Assume  $(\mathcal{E}^n, \mathcal{D}(\mathcal{E}^n))_{n\in \mathbb{N}}$ and $(\mathcal{E}, \mathcal{D}(\mathcal{E}))$  are  quadratic forms with dense domains in $L^2(E,\mu)$ where $(E,\mu ) $ is a measure space. One says that the sequence $(\mathcal{E}^n, \mathcal{D}(\mathcal{E}^n))_{n\in \mathbb{N}}$ converges in $L^2(E,\mu)$  in the Mosco sense  to $(\mathcal{E}, \mathcal{D}(\mathcal{E}))$ if the following two conditions are satisfied.

\medskip

\noindent \textbf{Limsup:} For every $u\in L^2(E,\mu)$ there exists a sequence $(u_n)_n$ in $ L^2(E,\mu)$ such that   $u_n\in  \mathcal{D}(\mathcal{E}^n)$,  $u_n\to u$ (read $u_n$ strongly converges to $u$) in  $ L^2(E,\mu)$ and 
\[\limsup_{n\to \infty} \mathcal{E}^n(u_n,u_n) \leq \mathcal{E}(u,u). \]
\textbf{Liminf:} For every sequence, $(u_n)_n$   with   $u_n\in  \mathcal{D}(\mathcal{E}^n)$ and every $u\in \mathcal{D}(\mathcal{E})$ such that   $u_n \rightharpoonup u$ (read $u_n$ weakly converges to $u$) in  $ L^2(E,\mu)$ we have, 
\[\mathcal{E}(u,u)\leq  \liminf_{n\to \infty}  \mathcal{E}^n(u_n,u_n).\] 
\end{definition}

\begin{remark}
(i) It is worth emphasizing that, combining the $\limsup$ and $\liminf$ conditions,   the  $\limsup$ condition is equivalent to the existence of a sequence $(u_n)_n$ in $ L^2(E,\mu)$ such that   $u_n\in  \mathcal{D}(\mathcal{E}^n)$,  $u_n\to u$  in $ L^2(E,\mu)$ and 
\[\lim_{n\to \infty} \mathcal{E}^n(u_n,u_n)=\mathcal{E}(u,u). \]
(ii) Also note that, replacing the weak convergence   in the $\liminf$ condition by the strong convergence, one recovers the  famous  concept of Gamma convergence. 
\end{remark}

\medskip

The  following Theorem is reminiscent of \cite[Theorem 2]{BBM01}. 
\begin{theorem}\label{thm:quadratic-convergence-BBM}
Let $D\subset \mathbb{R}^d$  be an open extension domain and bounded. Then, under assumptions \eqref{eq:elliptic-condition} and \eqref{eq:integrability-condition} we have
\begin{align}\label{eqquadratics-limit}
 \lim_{\alpha \to 2^-}\iil_{D D} (u(x)-u(y))^2 J^\alpha(x,y)\d x \, \d y = \il_{D} \langle A(x) \nabla u(x), \nabla u(x) \rangle \d x.
\end{align}
for all $u\in  H^{1}(D)$. In particular, if $J^\alpha= 2d|x-y|^{-2}\rho_{2-\alpha}(x-y)$ or $J^\alpha= J_k^\alpha$ with $k=1,2,3$ then 

\begin{align*}\label{eqquadratics-limit-special}
 \lim_{\alpha \to 2^-} \frac{1}{2}\iil_{D D} (u(x)-u(y))^2 J^\alpha(x,y)\d x \, \d y = \il_{D} |\nabla u(x) |^2 \d x.
\end{align*}
\end{theorem}
\noindent In the proof we will make use of the following simple observation.

\begin{lemma}\label{lem:concentration}
Assume $\beta\geq 0$ and $R>0$. Then, obviously, $\int_{|x|\leq R}\rho_{2-\alpha}(x)\d x \leq 1$. Moreover, 
\begin{align*}
\lim\limits_{\alpha\to 2^-} \int_{|x|\leq R}|x|^\beta\rho_{2-\alpha}(x)\d x = 
\begin{cases}
1 \quad &\text{ if } \beta=0 \,, \\
0 &\text{ if }\beta > 0 \,.
\end{cases}
\end{align*}
\end{lemma}

\begin{proof}[Proof of \autoref{thm:quadratic-convergence-BBM}] \autoref{lem-cont-qua-form} suggests that  it suffices to prove \eqref{eqquadratics-limit}  for   $u$ in a  dense subset of $H^{1}(D)$. For instance, let us choose $u\in C^2(\overline{D}) $. 
 
 \begin{align} 
 &\iil_{D\times  D \cap \{|x-y|\geq 1\}} (u(x)-u(y))^2 J^\alpha(x,y)\d x \, \d y \nonumber\\
 & \quad \leq  4\il_{D}u^2(x) \d x \il_{ |x-y|\geq 1} J^\alpha(x,y) \, \d y \to 0, ~\hbox{as $\alpha\to 2^-$} \,. 
\end{align}
\noindent Now, we consider the mapping $F:D\times (0,2)\to \mathbb{R}$ with
 \begin{align*}
 F(x,\alpha):= \il_{|x-y|\leq 1} (u(x)-u(y))^2 J^\alpha(x,y) \d y.
 \end{align*} 
  By Taylor expansion we obtain
 \begin{align*}
 u(y)-u(x) = \nabla u (x)\cdot(y-x)+ r_1(x,y)|x-y|^2
 \end{align*}
therefore,  we can write 
\begin{align*}
 (u(y)-u(x))^2 = (\nabla u (x)\cdot(y-x))^2+ r(x,y)|x-y|^3
 \end{align*} 
 with bounded remainders $r(x,y)$  and $r_1(x,y)$. Hence, $F(x,\alpha)$ can be written as  
  \begin{align*}
  F(x,\alpha)&= \il_{|x-y|\leq 1}  [\nabla u(x)\cdot (y-x)]^2 J^\alpha(x,y) \d y+ R(x,\alpha)\,. 
\end{align*}
with
\begin{align*}
|R(x,\alpha)| &:=\Big|\,  \il_{|x-y|\leq 1} r(x,y) |x-y|^3 J^\alpha(x,y) \d y\Big| \\
&\leq   C  \il_{|x-y|\leq 1 } |x-y|\rho_{2-\alpha}(x-y)\d y \to 0 ~~\mbox{ as } \alpha \to 2^- \,.
\end{align*}
Here, we have applied \eqref{eq:elliptic-condition} and \autoref{lem:concentration}. Finally, we obtain
\begin{align*}
 \lim_{\alpha\to 2^-} F(x,\alpha) &= \lim_{\alpha\to 2^-} \il_{|x-y|\leq 1}  [\nabla u(x)\cdot (y-x)]^2  J^\alpha(x,y) \d y \\
 &=  \lim_{\alpha\to 2^-} \il_{|x-y|\leq 1} \left[\sum_{i=1}^{d} \partial_i u(x)  (y_i-x_i)\right]^2  J^\alpha(x,y) \d y \\
 &= \sum_{0\leq i,j\leq d}   \partial_i u(x)  \partial_j u(x)  \lim_{\alpha\to 2^-} \il_{|x-y|\leq 1} ( y_i-x_i)(y_j-x_j) J^\alpha(x,y) \d y \\
 &= \sum_{0\leq i,j\leq d}  a_{ij} (x)\partial_i u(x)  \partial_j u(x) = \langle A (x)\nabla u(x), \nabla u(x)\rangle.
\end{align*}
In particular,  if $J^\alpha(x,y)\mathbbm{1}_{B_1}(x-y)= \frac{C_{d,\alpha }}{2}|x-y|^{-d-\alpha}$ then,  thanks to the rotationally invariance of the Lebesgue measure  we get $a_{ij}(x) =0$ for $i\neq j$ and 

 \begin{align*}
a_{ii}(x)&=   \lim_{\alpha\to 2^-}\frac{C_{d,\alpha}}{2}  \il_{|x-y|\leq 1} h_i^2 |h|^{-d-\alpha} \d h=\lim_{\alpha\to 2^-}  \frac{C_{d,\alpha}}{2d} \il_{|x-y|\leq 1} |h|^{2-d-\alpha} \d h\\
 %
 %
 &=\lim_{\alpha\to 2^-}  \frac{C_{d,\alpha}}{2d \omega_{d-1} (2-\alpha)}  = 1.
 \end{align*}
 The fact that,  $  \frac{C_{d,\alpha}}{2d \omega_{d-1} (2-\alpha)}  \to  1$ can be found in  \cite{Hitchhiker}. Similar conclusion also  holds if $J^\alpha(x,y)\mathbbm{1}_{B_1}(x-y)=d| x-y|^{-2}\rho_{2-\alpha}(x-y)$.    Now noticing that the function $F(x, \alpha) $ is bounded on  $D\times (0,2) $, the Lebesgue's dominated convergence theorem yields

 \begin{align*}
 \lim_{\alpha \to2^-} \hspace*{-3ex} \iil_{D\times D\cap \{|x-y|\leq1 \}} \hspace*{-2ex} (u(x)-u(y))^2J^\alpha(x,y)\d x \, \d y = \lim_{\alpha \to2^-} \il_{D} F(x,\alpha)\d x = \il_{D} \langle A(x)\nabla u(x), \nabla u(x) \rangle \ \d x.
 \end{align*}
Altogether, we obtain  the required result. 
\end{proof}

\begin{lemma}\label{lem-liminf}
Let $\Omega$ be a bounded and open subset  of  $\mathbb{R}^d$. Assume $(u_n)_{n}\subset L^2(\Omega) $ is a  sequence converging in $L^2(\Omega)$ to some $u \in H^{1}(\Omega)$. 
Then, under the assumptions $\eqref{eq:elliptic-condition}$ and \eqref{eq:translation-invariance}, for any given  sequence $\alpha_n\in (0,2)$ such that $\alpha_n\to 2^- $ we have
\begin{align}\label{eq:convex-ineq}
\il_{\Omega} \langle A\nabla u(x), \nabla u(x) \rangle  \d x  \leq \liminf\limits_{n \to \infty} \iil_{\Omega\Omega} (u_{n}(x) -u_{n}(y))^2 J^{\alpha_n}(x,y)\d  x\d y.
\end{align}
\end{lemma}
\begin{proof} 
We borrow the technique from \cite{Brezis-const-function} and it is worth mentioning that an inequality similar to \eqref{eq:convex-ineq} appears in \cite{Ponce2004}. Assume $0 \in \Omega$ otherwise one can consider any arbitrary point $x_0$ in $\Omega$. Let us fix  $\delta >0$ small enough  and put, $\Omega_\delta =\{x\in \Omega: \operatorname{dist}(x,\partial\Omega)>\delta\}$. 
Let  consider $\phi \in C_c^{\infty}(\mathbb{R}^d)$ supported in $B_1(0)$ be such that $\phi \geq 0$ and $ \int_{} \phi = 1$.  
Define the mollifier $\phi_\delta(x)= \frac{1}{\delta^d}\phi\left(\frac{x}{\delta}\right)$ with support in $B_\delta(0)$   and let $u^\delta_{n} = u_{n} *\phi_\delta$  denote the convolution of $u_n $ and $\phi_\delta$. For sake of the simplicity we will assume $u_n,$ and $u$ are extended by zero outside of $\Omega$.
Assume $z\in\Omega_\delta  $ and $|h|\le \delta$ then, $z-h\in \Omega_\delta-h \subset \Omega$ so that,  the translation invariance condition \eqref{eq:translation-invariance} implies,
 \begin{align*}
 \iil_{\Omega_\delta\Omega_\delta} \left( u_{n}(x-h) -u_{n}(y-h)\right)^2 J^{\alpha_n}(x,y)\d  x\d y \leq  \iil_{\Omega\Omega}(u_{n}(x) -u_{n}(y))^2 J^{\alpha_n}(x,y)\d  x\d y.  
 \end{align*}
Thus given that, $\int_{} \phi_\delta = 1$, integrating both side over the ball $B_\delta(0)$ with respect to 
$\phi_\delta(h)dh$ and employing Jensen's inequality afterwards, yields 
\begin{align}\label{eqmolification-convex-Jessen}
 \iil_{\Omega_\delta\Omega_\delta} \left(u^\delta_{n}(x) -u^\delta_{n}(y)\right)^2 J^{\alpha_n}(x,y)\d  x\d y \leq  \iil_{\Omega\Omega} \left( u_{n}(x) -u_{n}(y)\right)^2 J^{\alpha_n}(x,y)\d  x\d y.  
\end{align}

By Lemma \ref{lem-cont-qua-form} there is a constant C independent on $\alpha_n$ for which, 
\begin{align*}
\left| (\mathcal{E}^{\alpha_n}_{\Omega_\delta}(u_n^\delta, u_n^\delta))^{1/2}-(\mathcal{E}^{\alpha_n}_{\Omega_\delta}(u^\delta, u^\delta))^{1/2}\right|
&\leq (\mathcal{E}^{\alpha_n}_{\Omega_\delta}(u_n^\delta-u^\delta, u_n^\delta-u^\delta))^{1/2}\\
&\leq  C\|u_n^\delta-u^\delta\|_{H^1(\Omega_\delta )}\\
&\leq C \|\phi_\delta\|_{W^{1,\infty}(B_\delta )}\|u_n-u\|_{L^2(\Omega )}.
\end{align*}
 Which implies,
  \begin{align*}
\left| (\mathcal{E}^{\alpha_n}_{\Omega_\delta}(u_n^\delta, u_n^\delta))^{1/2}-(\mathcal{E}^{\alpha_n}_{\Omega_\delta}(u^\delta, u^\delta))^{1/2}\right|\leq C \|\phi_\delta\|_{W^{1,\infty}(B_\delta )}\|u_n-u\|_{L^2(\Omega )} \to 0.
\end{align*}
 since by assumption,  $\|u_{n}-u\|_{L^2(\Omega)}\to 0$.  On the other hand, Theorem \ref{thm:quadratic-convergence-BBM} yields  that,  $\mathcal{E}^{\alpha_n}_{\Omega_\delta}(u^\delta, u^\delta)\to  \mathcal{E}^{A}_{\Omega_\delta}(u^\delta, u^\delta)$. Thus, we have shown that 
 \begin{align*}
 \mathcal{E}^{\alpha_n}_{\Omega_\delta}(u_n^\delta, u_n^\delta) \to  \mathcal{E}^{A}_{\Omega_\delta}(u^\delta, u^\delta).
\end{align*}  

Inserting this in \eqref{eqmolification-convex-Jessen}, we obtain

\begin{align*}
\il_{\Omega_\delta} \langle A\nabla u^\delta(x), \nabla u^\delta(x) \rangle  \d x  \leq \liminf\iil_{\Omega\Omega} (u_{n}(x) -u_{n}(y))^2 J^{\alpha_n}(x,y)\d  x\d y.
\end{align*}
Given that $u \in H^{1}(\Omega)$, it is clear that  $\nabla(\phi_\delta*u)=\phi_\delta* \nabla u$ and hence the desired inequality follows by letting $\delta \to 0^+$ since $\| \phi_\delta* \nabla u - \nabla u \|_{L^2(\Omega)}\to 0$ as $\delta \to 0^+$.
\end{proof}

Finally, we now are in the position to prove our main result, \autoref{thm:Mosco-convergence}.

\begin{proof}[Proof of \autoref{thm:Mosco-convergence}] Note that $C_c^\infty(\mathbb{R}^d) \subset V_{\nu^\alpha}(\Omega|\R^d)$  and $V_{\nu^\alpha}(\Omega|\R^d) \big|_\Omega \subset   H_{\nu^{\alpha}}( \Omega)   \subset L^{2}( \Omega) $. Hence  the denseness of domains in $L^2(\Omega)$ readily follows from \autoref{thm:density}. We consider the "$\limsup$" and the "$\liminf$"-part separately. 

\medskip	 

\par \textbf{Limsup:}
Let $u\in L^2(\Omega)$, if  $u \not\in H^1(\Omega)$ then the $\limsup$ statement holds true since $ \mathcal{E}^A(u,u)=\infty$. Now if  $u \in H^1(\Omega)$. By identifying  $u$ to one of its   extension $\overline{u}\in  H^1(\mathbb{R}^d)$, for sake of simplicity we can always assume that $u \in H^1(\mathbb{R}^d) $. On the one hand, Theorem \ref{thm:quadratic-convergence-BBM} shows that $\lim\limits_{\alpha\to 2^{-}} \mathcal{E}^\alpha_{\Omega}(u,u)=\mathcal{E}^A(u,u)$. On the other hand,  since by Theorem \ref{thm:density}, $C_c^\infty(\mathbb{R}^d)$ is dense in $H^1(\mathbb{R}^d)\cap V_{\nu^\alpha}(\Omega|\R^d)$ and 
\begin{align*}
 \mathcal{E}^\alpha(u,u)=  \mathcal{E}^\alpha_{\Omega}(u,u) + 2\iil_{\Omega\Omega^c} (u(x)-u(y))^2J^\alpha(x,y)\d x\d y
\end{align*}
it remains to show that,  for $ u \in C_c^\infty(\mathbb{R}^d)$
\begin{align}\label{eqextern-lim}
 \iil_{\Omega\Omega^c} (u(x)-u(y))^2J^\alpha(x,y)\d x\d y\to 0,\quad \text{as $\alpha\to 2^-$.}
\end{align}
To this end, let us assume $u\in C_c^\infty(\mathbb{R}^d)$.  Then we have  
\[ |u(y)-u(x)|^2  \leq \|\nabla u\|^2_{\infty} | x-y|^2. \]

Let $R>0$ large enough such that, $\operatorname{supp} u\subset B_{R/2}(0)$ and for fix $x\in \Omega$, let   $\delta_x= dist(x, \partial \Omega)>0 $, we obtain the following estimates
%

\begin{align*}
 \int_{\Omega^c} \frac{(u(x) -u(y))^2}{|x-y|^{2}} \rho_{2-\alpha}(x-y) dy &=\hspace*{-2ex}  \int\limits_{R>|x-y|> \delta_x} \hspace*{-3ex} \frac{(u(x) -u(y))^2}{|x-y|^{2}} \rho_{2-\alpha}(x-y)\d y+  u^2(x) \hspace*{-2ex} \int\limits_{|x-y|\geq  R} \hspace*{-2ex} \frac{ \rho_{2-\alpha}(x-y) }{|x-y|^{2}} \d y\\
  &\leq  \|\nabla u\|^2_{\infty}\hspace*{-2ex}  \int\limits_{R>|x-y|> \delta_x} \rho_{2-\alpha}(x-y)\d y+  \|u\|^2_{\infty} R^2 \hspace*{-2ex} \int\limits_{|x-y|\geq  R} \hspace*{-2ex}  \rho_{2-\alpha}(x-y)  \d y\\
 &\leq   C \int\limits_{|x-y|> \delta_x} \rho_{2-\alpha}(x-y) dy \to 0 \quad \text{ as } \alpha\to 2^{-}.
\end{align*}
Moreover, from the above estimates one also has, 
\begin{align*}
 \int_{\Omega^c} \frac{(u(x) -u(y))^2}{|x-y|^{2}} \rho_{2-\alpha}(x-y) dy \leq   C\hspace*{-2ex}\int\limits_{|x-y|>\delta_x} \rho_{2-\alpha}(x-y) dy \leq C
\end{align*}
with the constant $C$  independent on $x$. Hence, combining this and the assumption \eqref{eq:elliptic-condition}, the statement \eqref{eqextern-lim} follows from the dominated convergence theorem. Thus, we conclude that for $u\in H^1(\Omega)$,
 \[\limsup_{\alpha \to 2^{-}} \mathcal{E}_\Omega^\alpha(u,u)= \limsup_{\alpha \to 2^{-}} \mathcal{E}^\alpha(u,u)=  \mathcal{E}^A(u,u). \]

  Thus, choosing the  constant sequence  $u_\alpha= u$ for all $ \alpha \in (0,2)$ we are provided with  the $\limsup$ condition for  both forms $( \mathcal{E}^\alpha_{\Omega}(\cdot, \cdot), H_{\nu^{\alpha}}( \Omega) )_{\alpha }$  and $( \mathcal{E}^\alpha(\cdot,\cdot) , V_{\nu^\alpha}(\Omega|\R^d))_{\alpha}$.

\vspace{.5cm}
\par \textbf{Liminf}:  Let $u, u_n\in L^2(\Omega)$ be  such that,   $u_n  \rightharpoonup u$ in $L^2(\Omega)$. Necessarily, $(u_n)_{n}$ is bounded in $ L^2(\Omega)$. 
Let $( \alpha_n)_n$ be a sequence in $(0,2)$ such that $ \alpha_n \to 2^-$ as $n\to \infty$.  If   $\liminf\limits_{n \to \infty} \mathcal{E}_\Omega^{\alpha_n}(u_n,u_n) =\infty$ then, 
\[  \mathcal{E}^A(u,u)\leq \liminf_{n \to \infty} \mathcal{E}_\Omega^{ \alpha_n }(u_n,u_n) = \liminf_{n \to \infty} \mathcal{E}^{ \alpha_n }(u_n,u_n) =\infty.  \]
Assume $\liminf\limits_{n \to \infty} \mathcal{E}_\Omega^{ \alpha_n }(u_n,u_n)<\infty$ then according to \cite{BBM01, Ponce2004}  the sequence $(u_n)_n$ has a subsequence  (which we again denote by $u_n$) converging  in $L^2(\Omega)$ to some $\widetilde{u}\in H^1(\Omega)$. Consequently, as  $u_n \rightharpoonup u$ it readily follows that, $u_n\to u$ in $L^2(\Omega)$. Therefore, taking into account that $u\in H^1(\Omega)$, the desired liminf inequality   is an immediate consequence of \autoref{lem-liminf}. The proof of \autoref{thm:Mosco-convergence} is complete.
\end{proof}

We adopt the convention that, for a given quadratic form $\big(\mathcal{E}, \mathcal{D}(\mathcal{E})\big)$, we have 
$\mathcal{E}(u,u)= \infty$ whenever $u \not\in  \mathcal{D}(\mathcal{E})$. The next result is a variant of \autoref{thm:Mosco-convergence} with $H^1(\Omega)$ replaced by $H^1_0(\Omega)$

\begin{theorem}\label{thm:Mosco-convergence-bis}
	Let $\Omega\subset \mathbb{R}^d$ be an open bounded set with a continuous boundary. Assume \eqref{eq:elliptic-condition}, \eqref{eq:integrability-condition} and \eqref{eq:translation-invariance}.
	Then the two families of  quadratic forms $\big(\mathcal{E}^\alpha_{\Omega}(\cdot, \cdot),\overline{C_c^\infty(\Omega)}^{ H_{\nu^\alpha}( \Omega)} \big)_{\alpha}$  and  
	$\big( \mathcal{E}^\alpha(\cdot, \cdot),V^\Omega_{\nu^\alpha}( \Omega|\mathbb{R}^d) \big)_{\alpha}$ both converge to  $( \mathcal{E}^A(\cdot,\cdot), H_0^{1}( \Omega) )$ in the Mosco sense in $L^2(\Omega)$ as $\alpha\to 2^-$.
\end{theorem}

The result relies on the density of $\overline{C_c^\infty(\Omega)}^{ H_{\nu^\alpha}( \Omega)}$ resp. $V^\Omega_{\nu^\alpha}( \Omega|\mathbb{R}^d)$. The density of the first space is trivial. The density of the second space is formulated in \autoref{thm:density-omega}. Apart from the density issue, the details of the proof are the same as in the proof of \autoref{thm:Mosco-convergence}.

\section{Examples  of kernels}\label{sec:examples}

Here we collect some concrete  examples of sequences  $(\rho_\varepsilon)_\varepsilon$ satisfying the assumptions in \autoref{def:nu-alpha}. Note that we have two different kinds  of examples. The functions $h \mapsto \rho_\eps(h)$ that appear in \autoref{ex:most-important} are unbounded and the singularity gets critical at $h=0$ as $\eps \to 0+$. The functions $\rho_\eps$ that appear in \autoref{ex:bounded-nu} are bounded where the bound depends on a rescaling factor that blows up as $\eps \to 0+$. Both examples lead to a diffusion operator resp. gradient form in the limit.

\medskip

\noindent Through all these examples, $d\geq 1$, the constant  $\omega_{d-1}$ is the area of the $d-1$-dimensional unit  sphere and  $\varepsilon_{0} >0$ is a fixed number. 

\medskip

\begin{example}\label{ex:4-0}
This example is taken from \autoref{ex:most-important}. For $\eps > 0$ and $x \in \R^d$ set 
\begin{align*}
\rho_\eps (x) = \frac{\eps}{\omega_{d-1}} |x|^{-d+\eps} \mathbbm{1}_{B_1}(x)\,.
\end{align*}
\end{example}

\begin{example}\label{ex:4-1}
This example is a version of \autoref{ex:bounded-nu}. Assume  $d\geq 1$,  $0<\varepsilon < \varepsilon_{0}$   and  any  $-d < \beta \leq 2$. Set 
\begin{align*}
	\rho_\varepsilon(x) = \frac{d+\beta}{ \omega_{d-1}\varepsilon^{d+\beta}} |x|^{\beta}\mathbbm{1}_{B_{\varepsilon}}(x),\qquad\qquad x\in \mathbb{R}^d. 
\end{align*} 
\end{example}

\begin{example}\label{ex:4-2}
For   $d\geq 1$ and  $0<\varepsilon < \varepsilon_{0}$. Set
\begin{align*}
\rho_\varepsilon(x) = \frac{1}{\omega_{d-1}\log(\varepsilon_{0}/\varepsilon )}|x|^{-d}\mathbbm{1}_{\{\varepsilon <|x|<\varepsilon_{0}\}}. 
\end{align*} 
Note that  $ \log(\varepsilon_{0}/\varepsilon )\sim  |\log(\varepsilon )| $. This example  is the counter part of \autoref{ex:4-1} at the end point $\beta =-d$.
\end{example}


\begin{example}\label{ex:4-3}
Assume   $d\geq 1$, $0<\varepsilon < \varepsilon_{0}$ and  $-d < \beta \leq 2$.  For $x \in \mathbb{R}^d$ consider
\begin{align*}
\rho_\varepsilon(x) = \frac{(|x|+\varepsilon)^{\beta}}{\omega_{d-1} b_\varepsilon}\mathbbm{1}_{B_{\varepsilon_{0}}}(x)\qquad\qquad \hbox{with }\quad   b_\varepsilon =  \varepsilon^{d+\beta}\int^{1}_{\frac{\varepsilon}{\varepsilon+\varepsilon_{0}}}  t^{-d-\beta-1}(1-t)^{d-1}dt. 
\end{align*} 
The constant $b_\varepsilon $ is chosen such that $\int_{\mathbb{R}^d} \rho_\varepsilon(x) dx =1$. Additionally one  can check     $$\frac{(d+\beta)}{b_\varepsilon\varepsilon_{0} ^{d+\beta}}\to 1\quad \hbox{as }~~~\varepsilon\to 0^+. $$
\end{example}

\begin{example}\label{ex:4-4}
Assume   $d\geq 1$ and  $0<\varepsilon < \varepsilon_{0}$.  For $x \in \mathbb{R}^d$ consider
\begin{align*}
\rho_\varepsilon(x) = \frac{(|x|+\varepsilon)^{-d}}{\omega_{d-1} b_\varepsilon}\mathbbm{1}_{B_{\varepsilon_{0}}(x)}\qquad\qquad \hbox{with }\quad   b_\varepsilon =  \int^{1}_{\frac{\varepsilon}{\varepsilon+\varepsilon_{0}}}  t^{-1}(1-t)^{d-1}dt. 
\end{align*}  
The choice of the constant $b_\varepsilon $ ensures $\int_{\mathbb{R}^d} \rho_\varepsilon(x) dx =1$. It is not difficult to check 
$$ \frac{|\log(\varepsilon)|}{b_\varepsilon}\to 1\quad \hbox{as }~~~\varepsilon\to 0^+. $$
This example is the  counter part  of \autoref{ex:4-3} at the end point $\beta=-d$.
\end{example}

\begin{example}\label{ex:4-5} Assume   $d\geq 1$, $0<\varepsilon < \varepsilon_{0}$ and $\beta>0$.  For $x \in \mathbb{R}^d$ consider

\begin{align*}
\rho_\varepsilon(x) = \frac{|x|^\beta}{\omega_{d-1} b_\varepsilon(|x|+\varepsilon)^{d+\beta}}\mathbbm{1}_{B_{\varepsilon_{0}}(x)}\qquad\qquad \hbox{with }\quad   b_\varepsilon =  \int^{1}_{\frac{\varepsilon}{\varepsilon+\varepsilon_{0}}}  t^{-1}(1-t)^{d+\beta-1}dt. 
\end{align*} 
As above, the choice of $b_\varepsilon $ ensures $\int_{\mathbb{R}^d} \rho_\varepsilon(x) dx =1$. It is not difficult to check that   
$$ \frac{|\log(\varepsilon)|}{b_\varepsilon}\to 1\quad \hbox{as }~~~\varepsilon\to 0^+. $$
\end{example}

\begin{example}\label{ex:4-6} Assume $d\geq 1$, $0<\varepsilon<\varepsilon_0$. Let $\phi: \mathbb{R}\to [0, \infty)$ be almost decreasing and such that, $\int_{\mathbb{R}}\phi(s)\d s= 1$
\begin{align*}
\rho_\varepsilon(x) = \frac{|x|^{-d+1}}{\omega_{d-1}\varepsilon }\phi\big(|x|/\varepsilon\big)\, .
\end{align*}
\end{example}


\newcommand{\etalchar}[1]{$^{#1}$}

\end{document}